\pdfoutput=1
\documentclass[a4paper,10pt]{article}
\usepackage[latin1]{inputenc}  
\usepackage[english,greek,frenchb,french]{babel}
\usepackage{url,csquotes}
\usepackage[hidelinks,hyperfootnotes=false]{hyperref}
\usepackage{float}
\usepackage{amsfonts,amssymb,enumerate}
\usepackage{amsmath}
\allowdisplaybreaks[1]
\usepackage{amsthm}
\usepackage{graphicx}
\usepackage{bbm}
\usepackage{listings}
\usepackage{titling}
\usepackage{dsfont}
\usepackage{color}

\newcommand{\E}{\mathbb{E}}
    \newcommand{\Prb}{\mathbb{P}}

	\DeclareMathOperator{\Animals}{Animals}	
		\newcommand{\sS}{\mathbb{S}}
	\DeclareMathOperator{\Diam}{Diam}

    \newcommand{\sZ}{\mathbb{Z}}
    \newcommand{\sC}{\mathcal{C}}
    \newcommand{\ep}{\varepsilon}
    \newcommand{\pp}{\mathfrak{p}}

		\newcommand{\cE}{\mathcal{E}}
			\newcommand{\cG}{\mathcal{G}}
		
	\newcommand{\cB}{\mathcal{B}}
		
		\newcommand{\cH}{\mathcal{H}}

	\newcommand{\sR}{\mathbb{R}}
	\newcommand{\sN}{\mathbb{N}}

 \theoremstyle{plain}  
\newtheorem{thm}{Theorem}[section]
\newtheorem{prop}{Proposition}[section]
\theoremstyle{plain}
\newtheorem{defn}{Definition}[section]
\newtheorem{lem}{Lemma}[section]
\newtheorem{cor}{Corollary}[section]
\newlength{\separationtitre}

\theoremstyle{remark}
\newtheorem{rk}{Remark}[section]

\date{}
\author{Barbara Dembin \thanks{LPSM UMR 8001, Université Paris Diderot, Sorbonne Paris Cité, CNRS, F-75013 Paris, France}}

\begin{document}

 \selectlanguage{english}

\title{Regularity of the time constant for a supercritical Bernoulli percolation \thanks{Research was partially supported by the ANR project PPPP (ANR-16-CE40-0016)}}

\maketitle

\textbf{Abstract}: We consider an i.i.d. supercritical bond percolation on $\sZ^d$, every edge is open with a probability $p>p_c(d)$, where $p_c(d)$ denotes the critical parameter for this percolation. We know that there exists almost surely a unique infinite open cluster $\sC_p$ \cite{Grimmett99}. We are interested in the regularity properties of the chemical distance for supercritical Bernoulli percolation. The chemical distance between two points $x,y\in\sC_p$ corresponds to the length of the shortest path in $\sC_p$ joining the two points. The chemical distance between $0$ and $nx$ grows asymptotically like $n\mu_p(x)$. We aim to study the regularity properties of the map $p\rightarrow\mu_p$ in the supercritical regime. 
This may be seen as a special case of first passage percolation where the distribution of the passage time is  $G_p=p\delta_1+(1-p)\delta_\infty$, $p>p_c(d)$. It is already known that the map $p\rightarrow\mu_p$ is continuous (see \cite{GaretMarchandProcacciaTheret}).
\newline



\section{Introduction}
The model of first passage percolation was first introduced by Hammersley and Welsh \cite{HammersleyWelsh} as a model for the spread of a fluid in a porous medium. Let $d\geq 2$. We consider the graph $(\sZ^d,\E^d)$ having for vertices $\sZ^d$ and for edges $\E^d$ the set of pairs of nearest neighbors in $\sZ^d$ for the Euclidean norm. To each edge $e\in\E^d$ we assign a random variable $t(e)$ with values in $\sR^+$ so that the family $(t(e),\,e\in\E^d)$ is independent and identically distributed according to a given distribution $G$. The random variable $t(e)$ may be interpreted as the time needed for the fluid to cross the edge $e$.  We can define a random pseudo-metric $T$ on this graph: for any pair of vertices $x$, $y\in\sZ^d$, the random variable $T(x,y)$ is the shortest time to go from $x$ to $y$. Let $x\in\sZ^d\setminus\{0\}$. One can ask what is the asymptotic behavior of the quantity $T(0,x)$ when $\|x\|$ goes to infinity. Under some assumptions on the distribution $G$, one can prove that asymptotically when $n$ is large, the random variable $T(0,nx)$ behaves like $n\cdot \mu_G(x)$ where $\mu_G(x)$ is a deterministic constant depending only on the distribution $G$ and the point $x$. The constant $\mu_G(x)$ corresponds to the limit of $T(0,nx)/n$ when $n$ goes to infinity, when this limit exists. This result was proved by Cox and Durrett in \cite{CoxDurrett} in dimension $2$ under some integrability conditions on $G$, they also proved that $\mu_G$ is a semi-norm. Kesten extended this result to any dimension $d\geq 2$ in \cite{Kesten:StFlour}, and he proved that $\mu_G$ is a norm if and only if $G(\{0\})<p_c(d)$. In the study of first passage percolation, $\mu_G$ is usually called the time constant. The constant $\mu_G(x)$ may be seen as the inverse of the speed of spread of the fluid in the direction of $x$.

It is possible to extend this model by doing first passage percolation on a random environment.  We consider an i.i.d. supercritical bond percolation on the graph $(\sZ^d,\E^d )$.  Every edge $e\in\E^d$ is open with a probability $p>p_c(d)$, where $p_c(d)$ denotes the critical parameter for this percolation. We know that there exists almost surely a unique infinite open cluster $\sC_p$ \cite{Grimmett99}. We can define the model of first passage percolation on the infinite cluster $\sC_p$. To do so, we consider a probability measure $G$ on $[0,+\infty]$ such that $G([0,\infty[)=p$. In this setting, the $p$-closed edges correspond to the edges with an infinite value and so the cluster $\sC_p$ made of the edges with finite passage time corresponds to the infinite cluster of a supercritical Bernoulli percolation of parameter $p$. The existence of a time constant for such distributions was first obtained in the context of stationary integrable ergodic field by Garet and Marchand in \cite{GaretMarchand04} and was later shown for an independent field without any integrability condition by Cerf and Théret in \cite{cerf2016}. 

The question of the continuity of the map $G\rightarrow \mu_G$ started in dimension $2$ with the article of Cox \cite{Cox}. He showed the continuity of this map under the hypothesis of uniform integrability: if $G_n$ weakly converges toward $G$ and if there exists an integrable law $F$ such that for all $n\in\sN$, $F$ stochastically dominates $G_n$, then $\mu_{G_n}\rightarrow \mu_{G}$. In \cite{CoxKesten}, Cox and Kesten prove the continuity of this map in dimension $2$ without any integrability condition. Their idea was to consider a geodesic for truncated passage times $\min(t(e),M)$, and along it to avoid clusters of $p$-closed edges, that is to say edges with a passage time larger than some $M>0$, by bypassing them with a short path in the boundary of this cluster. Note that by construction, the edges of the boundary have passage time smaller than $M$. Thanks to combinatorial considerations, they were able to obtain a precise control on the length of these bypasses.  This idea was later extended to all the dimensions $d\geq 2$ by Kesten in \cite{Kesten:StFlour}, by taking a $M$ large enough such that the percolation of the edges with a passage time larger than $M$ is highly subcritical: for such a $M$, the size of the clusters of $p$-closed edges can be controlled. However, this idea does not work anymore when we allow passage time to take infinite values.
In \cite{GaretMarchandProcacciaTheret}, Garet, Marchand, Procaccia and Théret proved the continuity of the map $G\rightarrow\mu_G$ for general laws on $[0,+\infty]$ without any moment condition. More precisely, let $(G_n)_{n\in\sN}$, and $G$ probability measures on $[0,+\infty]$ such that $G_n$ weakly converges toward $G$ (we write $G_n\overset{d}{ \rightarrow} G$), that is to say for all continuous bounded functions $f:[0,+\infty]\rightarrow [0,+\infty)$, we have 
$$\lim_{n\rightarrow +\infty} \int _{[0,+\infty]} fdG_n= \int _{[0,+\infty]} fdG\, .$$
Equivalently, we say that $G_n\overset{d}{ \rightarrow} G$ if and only if $\lim_{n\rightarrow +\infty}G_n([t,+\infty])=G([t,+\infty])$ for all $t\in [0,+\infty]$ such that $x\rightarrow G([x,+\infty])$ is continuous at $t$.
If moreover for all $n\in\sN$, $G_n([0,+\infty))>p_c(d)$ and $G([0,+\infty))>p_c(d)$, then
\begin{align*}
\lim_{n\rightarrow \infty} \sup_{x\in\mathbb{S}^{d-1}} |\mu_{G_n}(x)-\mu_G(x)|=0\, 
\end{align*}
where $\mathbb{S}^{d-1}$ is the unit sphere for the Euclidean norm.

In this paper, we focus on distributions of the form $G_p=p\delta_1+(1-p)\delta_\infty$, $p>p_c(d)$.  We denote by $\sC'_p$ be the subgraph of $\sZ^d$  whose edges are open for the Bernoulli percolation of parameter $p$. The travel time given a law $G_p$ between two points $x$ and $y\in\sZ^d$ coincides with the so-called chemical distance that is the graph distance between $x$ and $y$ in $\sC'_p$. Namely, for $x,y\in\sZ^d$, we define the chemical distance \smash{$D^{\sC'_p}(x,y)$ }as the length of the shortest $p$-open path joining $x$ and $y$. Note that if $x$ and $y$ are not in the same cluster of  $\sC'_p$, \smash{$D^{\sC'_p}(x,y)=+\infty$}. Actually, when $x$ and $y$ are in the same cluster,\smash{ $D^{\sC'_p}(x,y)$} is of order $\|y-x\|_1$. In \cite{AntalPisztora}, Antal and Pisztora obtained the following large deviation upper bound:
\begin{align*}
\limsup\limits_{\|y\|_1\rightarrow\infty}\frac{1}{\|y\|_1}\log\Prb[0\leftrightarrow y,D^{\sC'_p}(0,y)>\rho]<0\, .
\end{align*} 
This result implies that there exists a constant $\rho$ depending on the parameter $p$ and the dimension $d$ such that
\begin{align*}
\limsup\limits_{\|y\|_1\rightarrow\infty}\frac{1}{\|y\|_1}D^{\sC'_p}(0,y)\mathds{1}_{0\leftrightarrow y}\leq \rho, \, \Prb_p\text{ a.s.}
\end{align*} 
These results were proved using renormalization arguments. They were improved later in \cite{GaretMarchand04} by Garet and Marchand, for the more general case of a stationary ergodic field. They proved that $D^{\sC'_p}(0,x)$ grows linearly in $\|x\|_1$. More precisely, for each $y\in\sZ^d\setminus\{0\}$, they proved the existence of a constant $\mu_p(y)$ such that
\begin{align*}
\lim\limits_{\substack{n\rightarrow\infty \\ 0\leftrightarrow ny}}\frac{D^{\sC'_p}(0,ny)}{n}=\mu_p(y) , \, \Prb_p\text{ a.s.}\, .
\end{align*}
The constant $\mu_p$ is called the time constant.
The map $p\rightarrow\mu_p$ can be extended to $\mathbb{Q}^d$ by homogeneity and to $\mathbb{R}^d$ by continuity. It is a norm on $\mathbb{R}^d$.
This convergence holds uniformly in all directions, this is equivalent of saying  that an asymptotic shape emerges. Indeed, the set of points that are at a chemical distance from $0$ smaller than $n$ asymptotically looks like $n\mathcal{B}_{\mu_p}$, where $\mathcal{B}_{\mu_p}$ denotes the unit ball associated with the norm $\mu_p$. In another paper \cite{GaretMarchand07}, Garet and Marchand studied the fluctuations of $D^{\sC'_p}(0,y)/\mu_p(y)$ around its mean and obtained the following large deviation result:
\begin{align*}
\forall \ep>0, \; \limsup_{\|x\|_1\rightarrow \infty} \frac{\ln\Prb_p\left(0\leftrightarrow x, \frac{D^{\sC'_p}(0,y)}{\mu_p(y)}\notin (1-\ep,1+\ep)\right)}{\|x\|_1}<0\, .
\end{align*}
In the same paper, they showed another large deviation result that, as a corollary, proves the continuity of the map $p\rightarrow \mu_p$ in $p=1$.
In \cite{GaretMarchand10}, Garet and Marchand obtained moderate deviations of the quantity $|D^{\sC'_p}(0,y)-\mu_p(y)|$. 
As a corollary of the work of  Garet, Marchand, Procaccia and Théret in \cite{GaretMarchandProcacciaTheret} we obtain the continuity of the map $p\rightarrow \mu_p$ in $(p_c(d),1]$. Our paper is a continuation of \cite{GaretMarchandProcacciaTheret}, our aim is to obtain better regularity properties for the map $p\rightarrow \mu_p$ than just continuity. We prove the following theorem.
\begin{thm}[Regularity of the time constant] \label{heart}
Let $p_0>p_c(d)$. There exists a constant $\kappa_d$ depending only on $d$ and $p_0$, such that for all $p\leq q$ in $[p_0,1]$
$$\sup_{x\in\mathbb{S}^{d-1}}|\mu_{p}(x)- \mu_{q}(x)|\leq \kappa_d (q-p)|\log(q-p)| \,.$$
\end{thm}
\noindent To study the regularity of the map $p\rightarrow \mu_p$, our aim is to control the difference between the chemical distance in the infinite cluster $\sC_p$ of a Bernoulli percolation of parameter $p>p_c(d)$ with the chemical distance in $\sC_q$ where $q\geq p$. The key part of the proof lies in the modification of a path. We couple the two percolations such that a $p$-open edge is also $q$-open but the converse does not necessarily hold. We consider a $q$-open path for some $q\geq p>p_c(d)$. Some of the edges of this path are $p$-closed, we want to build upon this path a $p$-open path by bypassing the $p$-closed edges. In order to bypass them, we use the idea of \cite{GaretMarchandProcacciaTheret} and we build our bypasses at a macroscopic scale. This idea finds its inspiration in the works of Antal and Pisztora \cite{Pisztora} and Cox and Kesten \cite{CoxKesten}. We have to consider an appropriate renormalization and we obtain a macroscopic lattice with good and bad sites. Good and bad sites correspond to boxes of size $2N$ in the microscopic lattice. We will do our bypasses using good sites at a macroscopic scale that will have good connectivity properties at a microscopic scale. The remainder of the proof consists in getting probabilistic estimates of the length of the bypass. In this article we improve the estimates obtained in \cite{GaretMarchandProcacciaTheret}. We quantify the renormalization to be able to give quantitative bounds on continuity. Namely, we give an explicit expression of the appropriate size of a $N$-box. We use the idea of corridor that appeared in the work of Cox and Kesten \cite{CoxKesten} to have a better control on combinatorial terms and derive a more precise control of the length of the bypasses than the one obtained in \cite{GaretMarchandProcacciaTheret}.

We recall that $\cB_{\mu_p}$ denotes the unit ball associated with the norm $\mu_p$. From  Theorem  \ref{heart}, we can easily deduce the following regularity of the asymptotic shapes.
\begin{cor}[Regularity of the asymptotic shapes]\label{cor1}
Let $p_0>p_c(d)$. There exists a constant $\kappa'_d$ depending only on $d$ and $p_0$, such that for all $p\leq q$ in $[p_0,1]$,
\begin{align*}
d_{\cH}(\cB_{\mu_q},\cB_{\mu_p})\leq  \kappa'_d (q-p)|\log(q-p) |
\end{align*}
where $d_\cH$ is the Hausdorff distance between non-empty compact sets of $\sR^d$.
\end{cor} 

Here is the structure of the paper. In section \ref{defs}, we introduce some definitions and preliminary results that are going to be useful in what follows. The section \ref{renormalization} presents the renormalization process and how we modify a $q$-open path to turn it into a $p$-open path and how we can control the length of the bypasses. In section \ref{goodbox} and \ref{ProbaEst}, we get probabilistic estimates on the length of the bypasses. Finally, in section \ref{estimate} we prove the main Theorem \ref{heart} and its Corollary \ref{cor1}.

\begin{rk}
The section \ref{renormalization} is a simplified version of the renormalization process that was already present in \cite{GaretMarchandProcacciaTheret}. The simplification comes from the fact that we are not interested in general distributions but only on distributions $G_p$ for $p>p_c(d)$ which have the advantage of taking only two values $1$ or $+\infty$. The original part of this work is the quantification of the renormalization and the combinatorial estimates of section \ref{ProbaEst}. 
\end{rk}

\section{Definitions and preliminary results}\label{defs}
Let $d\geq 2$. Let us recall the different distances in $\mathbb{R}^d$. Let $x=(x_1,\dots,x_d)\in\mathbb{R}^d$, we define $$\|x\|_1=\sum_{i=1}^d |x_i|,\quad\|x\|_2=\sqrt{\sum_{i=1}^d x_i^2}\quad\text{and}\quad\|x\|_\infty=\max\{ |x_i|,i=1,\dots,d\}\,.$$ 

Let $\cG$ be a subgraph of $(\sZ^d,\E^d)$ and $x,y\in\cG$. A path $\gamma$ from $x$ to $y$ in $\cG$ is a sequence $\gamma=(v_0,e_1,\dots,e_n,v_n)$ such that $v_0=x$, $v_n=y$ and for all $i\in\{1,\dots,n\}$, the edge $e_i=\langle v_{i-1},v_i\rangle$ belongs to $\cG$. We say that $x$ and $y$ are connected in $\cG$ if there exists such a path. We denote by $|\gamma|=n$ the length of $\gamma$. We define $$D^{\cG}(x,y)=\inf\{|r|: \text{$r$ is a path from $x$ to $y$ in $\cG$}\}$$ the chemical distance between $x$ and $y$ in $\cG$. If $x$ and $y$ are not connected in $\cG$, $D^{\cG}(x,y)=\infty$. In the following, $\cG$ will be $\sC'_p$ the subgraph of $\sZ^d$ whose edges are open for the Bernoulli percolation of parameter $p>p_c(d)$.
To get around the fact that the chemical distance can take infinite values we introduce regularized chemical distance. Let $\sC\subset \sC'_p$ be a connected cluster, we define $\widetilde{x}^\sC$  as the vertex of $\sC$ which minimizes $\|x-\widetilde{x}^\sC\|_1$ with a deterministic rule to break ties. As $\sC\subset \sC'_p$, we have $$D^{\sC'_p}(\widetilde{x}^\sC,\widetilde{y}^\sC)\leq D^{\sC}(\widetilde{x}^\sC,\widetilde{y}^\sC)<\infty\,.$$ Typically, $\sC$ is going to be the infinite cluster for Bernoulli percolation with a parameter $p_0\leq p$ (thus $\sC_{p_0}\subset \sC'_p$).

We can define the regularized time constant as in \cite{GaretMarchand10} or as a special case of \cite{cerf2016}.
\begin{prop}\label{convergence}Let $p>p_c(d)$. There exists a deterministic function $\mu_p:\sZ^d\rightarrow [0,+\infty)$, such that for every $p_0\in(p_c(d),p]$:
\begin{align*}
\forall x\in\sZ^d\, \lim_{n\rightarrow \infty} \frac{D^{\sC_{p}}(\widetilde{0}^{\sC_{p_0}},\widetilde{nx}^{\sC_{p_0}})}{n}=\mu_p(x)\text{ a.s. and in $L^1$.}
\end{align*}
\end{prop}
\noindent It is important to check that $\mu_p$ does not depend on $p_0$, \textit{i.e.}, on the cluster  $\sC_{p_0}$ we use to regularize. This is done in Lemma 2.11 in \cite{GaretMarchandProcacciaTheret}. As a corollary, we obtain the monotonicity of the map $p\rightarrow \mu_p$ which is non increasing, see Lemma 2.12 in \cite{GaretMarchandProcacciaTheret}.
\begin{cor}\label{decmu}For all $p_c(d)<p\leq q$ and for all $x\in\sZ^d$,
$$\mu_p(x)\geq\mu_q(x)\,.$$
\end{cor} 

We will also need this other definition of path that corresponds to the context of site percolation. Let $\cG$ be a subset of $\sZ^d$ and $x,y\in\cG$. We say that the sequence $\gamma=(v_0,\dots,v_n)$ is a $\sZ ^d$-path from $x$ to $y$ in $\cG$ if $v_0=x$, $v_n=y$ and for all $i\in\{1,\dots,n\}$, $v_i\in\cG$ and $\|v_i-v_{i-1}\|_1=1$. 

\section{Modification of a path}\label{renormalization}

In this section we present the renormalization process. We are here at a macroscopic scale, we define good boxes to be boxes with useful properties to build our modified paths.

\subsection {Definition of the renormalization process} Let $p>p_c(d)$ be the parameter of an i.i.d. Bernoulli percolation on the edges of $\sZ^d$. For a large integer $N$, that will be chosen later, we set $B_N=[-N,N[^d\cap \sZ^d$ and define the following family of $N$-boxes, for $\textbf{i}\in \sZ^d$,
$$B_N(\textbf{i})=\tau_{\textbf{i}(2N+1)}(B_N)$$
where $\tau_b$ denotes the shift in $\sZ^d$ with vector $b\in\sZ^d$. $\sZ^d$ is the disjoint union of this family: $\sZ^d=\sqcup_{\textbf{i}\in\sZ^d}B_N(\textbf{i})$. We need to introduce larger boxes that will help us to link $N$-boxes together. For $\textbf{i}\in \sZ^d$, we define
$$B'_N(\textbf{i})=\tau_{\textbf{i}(2N+1)}(B_{3N}).$$
To define what a good box is, we have to list properties that a good box should have to ensure that we can build a modification of the path as we have announced in the introduction. We have to keep in mind that all the properties must occur with probability close to $1$ when $N$ goes to infinity. Before defining what a good box is, let us recall some definitions.
A connected cluster $C$ is crossing for a box $B$, if for all $d$ directions, there is an open path in $C\cap B$ connecting the two opposite faces of $B$.
We define the diameter of a finite cluster $\sC$ as 
$$\Diam(\sC):=\max_{\substack{i=1,\dots, d\\ x,y\in \sC}}|x_i-y_i|\, .$$
\begin{defn}\label{defbon} We say that the macroscopic site $\textbf{i}$ is $p$-good if the following events occur:
\begin{enumerate}[(i)]
\item There exists a unique $p$-cluster $\sC$ in $B'_N(\textbf{i})$ with diameter larger than $N$;
\item This $p$-cluster $\sC$ is crossing for each of the $3^d$ $N$-boxes included in $B'_N(\textbf{i})$;
\item For all $x,y\in B'_N(\textbf{i})$, if $x$ and $y$ belong to $\sC$ then $D^{\sC'_p}(x,y)\leq 12\beta N$, for an appropriate $\beta$ that will be defined later.
\end{enumerate}
$\sC$ is called the crossing $p$-cluster of the $p$-good box $B_N(\textbf{i})$.
\end{defn}

Let us define a percolation by site on the macroscopic grid given by the state of the boxes, \textit{i.e.}, we say that a macroscopic site $\textbf{i}$ is open if the box $B_N(\textbf{i})$ is $p$-good, otherwise we say the site is closed. Note that the state of the boxes are not independent, there is a short range dependence. 

On the macroscopic grid $\sZ^d$, we consider the standard definition of closest neighbors, that is to say $x$ and $y$ are neighbors if $\|x-y\|_1=1$. Let $C$ be a connected set of macroscopic sites, we define its exterior vertex boundary 
$$\partial_v C =\left\{\begin{array}{c} \textbf{i}\in\sZ^d\setminus C: \text {\textbf{i} has a neighbour in C and is connected } \\ \text{to infinity by a $\sZ^d$-path in $\sZ^d\setminus C$} \end{array} \right\}.$$
For a bad macroscopic site $\textbf{i}$, let us denote by $C(\textbf{i})$ the connected cluster of bad macroscopic sites containing $\textbf{i}$. If $C(\textbf{i})$ is finite, the set $\partial_v C (\textbf{i})$ is not connected in the standard definition but it is with a weaker definition of neighbors. We say that two macroscopic sites $\textbf{i}$ and $\textbf{j}$ are $*$-neighbors if and only if $\|\textbf{i}-\textbf{j}\|_\infty=1$. Therefore, $\partial_v C (\textbf{i})$  is an $*$-connected set of good macroscopic sites see for instance Lemma 2 in \cite{Timar}. 
We adopt the convention that $\partial_v C (\textbf{i})=\{\textbf{i}\}$ when $\textbf{i}$ is a good site.

\subsection{Construction of bypasses}
Let us consider $p_c(d)<p\leq q$, we fix $N$ in this section. Let us consider a $q$-open path $\gamma$. In this paper, we will consider two different couplings. We do not specify here what coupling we use. However, for these two couplings a $p$-open edge is necessarily $q$-open. Thus, some edges in $\gamma$ might be $p$-closed. We denote by $\gamma_o$ the set of $p$-open edges in $\gamma$, and by $\gamma_c$ the set of $p$-closed edges in $\gamma$. Our aim is to build a bypass for each edge in $\gamma_c$ using only  $p$-open edges. The proof will follow the proof of Lemma 3.2 in \cite{GaretMarchandProcacciaTheret} up to some adaptations.

As the bypasses are going to be made at a macroscopic scale, we need to consider the $N$-boxes that $\gamma$ crosses. We denote by $\Gamma\subset \sZ^d$  the connected set of all the $N$-boxes visited by $\gamma$. The set $\Gamma$ is connected in the standard definition. We denote by $Bad$ the random set of bad connected components on the macroscopic percolation given by the states of the $N$-boxes. The following Lemma states that we can bypass all the $p$-closed edges in $\gamma$ and gives a control on the total size of these bypasses.

\begin{lem}\label{lem1}
Let us consider $y,z\in\sC_p$  such that the $N$-boxes of $y$ and $z$ belong to an infinite cluster of $p$-good boxes. Let us consider a $q$-open path $\gamma$ joining $y$ to $z$. Then there exists a $p$-open path $\gamma'$ between $y$ and $z$ that has the following properties:
\begin{enumerate}[(1)]
\item $\gamma'\setminus \gamma$ is a set of disjoint self avoiding $p$-open paths that intersect $\gamma'\cap\gamma$ at their endpoints;
\item $|\gamma'\setminus \gamma|\leq \rho_dN\left(\sum\limits_{C\in Bad:C\cap\Gamma\neq\emptyset} |C| + |\gamma_c|\right)$, where $\rho_d$ is a constant depending only on the dimension $d$.
\end{enumerate}
\end{lem}

\begin{rk}
Note that here we don't need to introduce a parameter $p_0$ and require that the bypasses are $p_0$ open as in \cite{GaretMarchandProcacciaTheret}. Indeed, this condition was required because finite passage times of edges were not bounded. This is the reason why it was needed in \cite{GaretMarchandProcacciaTheret} to bypass $p$-closed edges with $p_0$-open edges. These $p_0$-open edges were precisely edges with passage time smaller than some constant $M_0$.  In our context, we can get rid of this technical aspect because passage times when finite may only take the value $1$.
\end{rk}

\noindent Before proving Lemma \ref{lem1}, we need to prove the following lemma that gives a control on the length of a path between two points in a $*$-connected set of good boxes.

\begin{lem} \label{lem2}Let $\mathcal{I}$ be a set of $n\in\mathbb{N}^*$ macroscopic sites such that $(B_N(\textbf{i}))_{\textbf{i}\in\mathcal{I}}$ is a $*$-connected set of $p$-good $N$-boxes. Let $x\in B_N(\textbf{j})$ be in the $p$-crossing cluster of $B_N(\textbf{j})$ with $\textbf{j}\in\mathcal{I}$ and $y\in B_N(\textbf{k})$ be in the $p$-crossing cluster of $B_N(\textbf{k})$ with $\textbf{k}\in\mathcal{I}$. Then, we can find a $p$-open path joining $x$ and $y$ of length at most $12\beta Nn$ (with the same constant $\beta$ as in Definition \ref{defbon}).

\end{lem}

\begin{proof}[Proof of Lemma \ref{lem2}]
Since $\mathcal{I}$ is a $*$-connected set of macroscopic sites, there exists a self-avoiding macroscopic $*$-connected path $(\varphi_i)_{1\leq i\leq r}\subset \mathcal{I}$  such that $\varphi_1=\textbf{j}$, $\varphi_r=\textbf{k}$. Thus, we get that $r\leq|\mathcal{I}| = n$. As all the sites in $\mathcal{I}$ are good, all the $N$-boxes corresponding to the sites $(\varphi_i)_{1\leq i\leq r}$ are good.

For each $2\leq i\leq r-1$, we define $x_i$ to be a point in the $p$-crossing cluster of the box $B_N(\varphi_i)$ chosen according to a deterministic rule. We define $x_1=x$ and $x_r=y$.
For each $1\leq i <r$, $x_i$ and $x_{i+1}$ both belong to $B'_N(\varphi_i)$. Using property $(iii)$ of a $p$-good box, we can build a $p$-open path $\gamma(i)$ from $x_i$ to $x_{i+1}$ of length at most $12\beta N$. By concatenating the paths $\gamma(1),\dots,\gamma(r-1)$ in this order, we obtain a $p$-open path joining $x$ to $y$ of length at most $ 12 \beta Nn$.

\end{proof}

\begin{proof}[Proof of Lemma \ref{lem1}] Let us consider $y,z\in\sC_p$  such that the $N$-boxes of $y$ and $z$ belong to an infinite cluster of $p$-good boxes. Let $\gamma$ be a $q$-open path joining $y$ to $z$. The idea is the following. We want to bypass all the $p$-closed edges of $\gamma$. Let us consider an edge $e\in\gamma_c$ and $B_N(\textbf{i})$ its associated $N$-box. There are two different cases:
\begin{itemize}
\item If $B_N(\textbf{i})$ is a good box, we can build a $p$-open bypass of $e$ at a microscopic scale by staying in a fixed neighborhood of $B_N(\textbf{i})$. We will use the third property of good boxes to control the length of the bypass that will be at most $12\beta N$.
\item If $B_N(\textbf{i})$ is a bad box, we must build a $p$-open bypass at a macroscopic scale in the exterior vertex boundary $\partial_v C (\textbf{i})$ that is an $*$-connected component of good boxes. We will use Lemma \ref{lem2} to control the length of this bypass.
\end{itemize}
\paragraph{}

Let $\varphi_0=(\varphi_0(j))_{1\leq j \leq r_0}$ be the sequence of $N$-boxes $\gamma$ visits. From the sequence $\varphi_0$, we can extract the sequence of $N$-boxes containing at least one $p$-closed edge of $\gamma$. We only keep the indices of the boxes containing the smallest extremity of a $p$-closed edge of $\gamma$ for the lexicographic order. We obtain a sequence  $\varphi_1=(\varphi_1(j))_{1\leq j\leq r_1}$. Notice that $r_1\leq r_0$ and $r_1\leq |\gamma_c|$.
Before building our bypasses, we have to get rid of some pathological cases. We are going to proceed to further extractions. Note that two $*$-connected components of $(\partial_v C (\varphi_1(j)))_{1\leq j\leq r_1}$ can be $*$-connected together, in that case they count as a unique connected component. Namely, the set $E=\cup_{1\leq j\leq r_1}\partial_v C (\varphi_1(j))$ has at most $r_1$ $*$-connected component $(S_{\varphi_2(j)})_{1\leq j\leq r_2}$. Up to reordering, we can assume that the sequence  $(S_{\varphi_2(j)})_{1\leq j\leq r_2}$ is  ordered in such a way that $S_{\varphi_2(1)}$ is the first $*$-connected component of $E$ visited by $\gamma$ among the $(S_{\varphi_2(j)})_{1\leq j\leq r_2}$, $S_{\varphi_2(2)}$ is the second and so on. 
\begin{figure}[h]
\def\svgwidth{1.4\textwidth}
 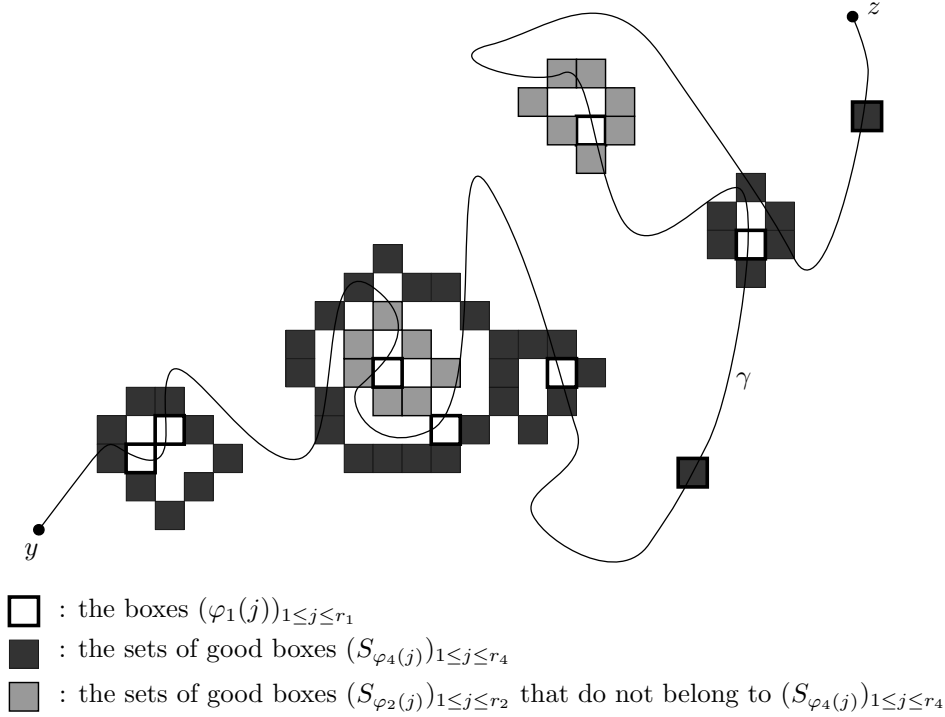
 \caption{Construction of the path $\gamma'$ - First step}
 \label{fig1}
\end{figure}
\noindent Next, we consider the case of nesting, that is to say when there exist $j\neq k$ such that $S_{\varphi_2(j)}$ is in the interior of $S_{\varphi_2(k)}$. In that case, we only keep the largest connected component $S_{\varphi_2(k)}$: we obtain another subsequence $(S_{\varphi_3(j)})_{1\leq j\leq r_3}$ with $r_3\leq r_2$. Finally, we want to exclude a last case, when between the moment we enter for the first time in a given connected component and the last time we leave this connected component, we have explored other connected components of $(S_{\varphi_3(j)})_{1\leq j\leq r_3}$. That is to say we want to remove the macroscopic loops $\gamma$ makes between different visits of the same $*$-connected components $S_{\varphi_3(j)}$ (see Figure \ref{fig1}).  We iteratively extract from $(S_{\varphi_3(j)})_{1\leq j\leq r_3}$ a sequence $(S_{\varphi_4(j)})_{1\leq j\leq r_4}$ in the following way: $S_{\varphi_4(1)}=S_{\varphi_3(1)}$, assume  $(S_{\varphi_4(j)})_{1\leq j\leq k}$ is constructed $\varphi_4(k+1)$ is the smallest indice $\varphi_3(j)$ such that $\gamma$ visits $S_{\varphi_3(j)}$ after its last visit to $S_{\varphi_4(k)}$. We stop the process when we cannot find such $j$. Of course, $r_4\leq r_3$. The sequence $(S_{\varphi_4(j)})_{1\leq j\leq r_4}$ is a sequence of sets of good $N$-boxes that are all visited by $\gamma$.

Let us introduce some notations (see Figure \ref{fig2}), we write $\gamma=(x_0,\dots,x_n)$. For all $k\in\{1,\dots,r_4\}$, we denote by $\Psi_{in}(k)$ (respectively $\Psi_{out}(k)$) the first moment that $\gamma$ enters in  $S_{\varphi_4(1)}$ (resp. last moment that $\gamma$ exits from  $S_{\varphi_4(1)}$). More precisely, we have $$\Psi_{in}(1)=\min\big\{\,j\geq 1, x_j\in S_{\varphi_4(1)}\,\big\}$$ and $$\Psi_{out}(1)=\max\big\{\,j\geq \Psi_{in}(1) , x_j\in S_{\varphi_4(1)}\,\big\}\,.$$ Assume $\Psi_{in}(1),\dots,\Psi_{in}(k)$ and $\Psi_{out}(1),\dots,\Psi_{out}(k)$ are constructed then $$\Psi_{in}(k+1)=\min\big\{\,j\geq \Psi_{out}(k), x_j\in S_{\varphi_4(k+1)}\,\big\}$$ and $$\Psi_{out}(k+1)=\max\big\{\,j\geq \Psi_{in}(k+1) , x_j\in S_{\varphi_4(k+1)}\,\big\}\,.$$ Let $B_{in}(j)$ be the $N$-box in $S_{\varphi_4(j)}$ containing $x_{\Psi_{in}(j)}$, $B_{out}(j)$ be the $N$-box in $S_{\varphi_4(j)}$ containing $x_{\Psi_{out}(j)}$. Let $\gamma(j)$ be the section of $\gamma$ from $x_{\Psi_{out}(j)}$ to $x_{\Psi_{in}(j+1)}$ for $1\leq j <r_4$, let $\gamma(0)$ (resp $\gamma(r_4)$) be the section of $\gamma$ from $y$ to $x_{\Psi_{in}(1)}$  (resp. from $x_{\Psi_{out}(r_4)}$ to $z$). 

We have to study separately the beginning and the end of the path $\gamma$. Note that as the $N$-boxes of $y$ and $z$ both belong to an infinite cluster of good boxes, their box cannot be nested in a bigger $*$-connected components of good boxes of the collection $(S_{\varphi_4(j)})_{1\leq j\leq r_4}$. Thus, if $B_N(\textbf{k})$, the $N$-box of $y$, contains a $p$-closed edge of $\gamma$, necessarily $S_{\varphi_4(1)}$ contains $B_N(\textbf{k})$, $B_{in}(1)=B_N(\textbf{k})$ and $x_{\Psi_{in}(1)}=y$. Similarly,  if $B_N(\textbf{l})$, the $N$-box of $z$, contains a $p$-closed edge of $\gamma$, necessarily $S_{\varphi_4(r_4)}$ contains $B_N(\textbf{l})$, $B_{out}(r_4)=B_N(\textbf{l})$ and  $x_{\Psi_{out}(r_4)}=z$. 

\begin{figure}[h]
\def\svgwidth{0.9\textwidth}
\begin{center}
 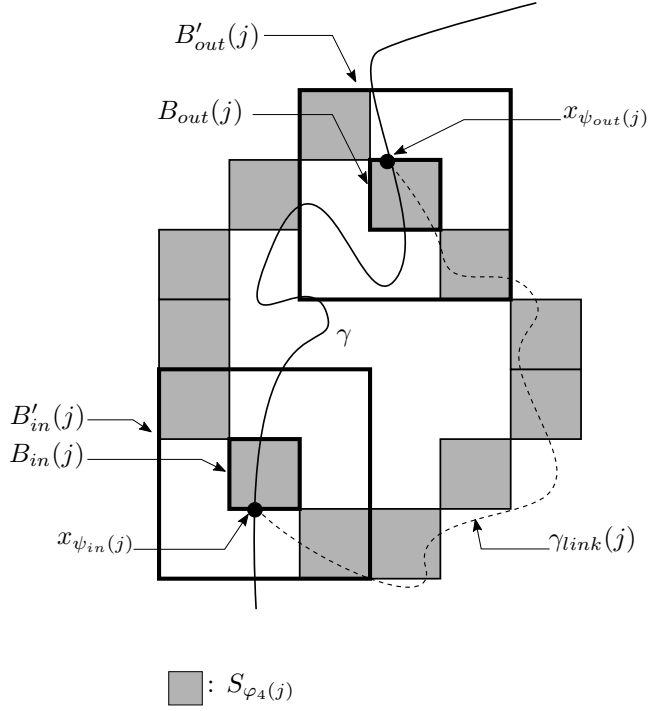
 \end{center}
 \caption{Construction of the path $\gamma'$ - Second step}
 \label{fig2}
\end{figure}
In order to apply Lemma \ref{lem2}, let us show that for every $j\in\{1,\dots,r_4\}$,  $x_{\Psi_{in}(j)}$ (resp. $ x_{\Psi_{out}(j)}$) belongs to the $p$-crossing cluster of $B_{in}(j)$ (resp. $B_{out}(j)$). Let us study separately the case of $x_{\Psi_{in}(1)}$ and  $x_{\Psi_{out}(r_4)}$. If $x_{\Psi_{in}(1)}=y$ then $x_{\Psi_{in}(1)}$ belongs to the $p$-crossing cluster of $B_{in}(j)$. Suppose that $x_{\Psi_{in}(1)}\neq y$. As $y\in \sC_p$ and $y$ is connected to $x_{\Psi_{in}(1)}$ by a $p$-open path, $x_{\Psi_{in}(1)}$ is also in $\sC_p$. By the property $(i)$ of a good box applied to $B_{in}(1)$, we get that $x_{\Psi_{in}(1)}$ is in the $p$-crossing cluster of $B_{in}(1)$. We study the case of $x_{\Psi_{out}(r_4)}$ similarly. 
To study $x_{\Psi_{in}(j)}$ (resp. $x_{\Psi_{out}(j)})$) for $j\in\{2,\dots,r_4-1\}$, we use the fact that by construction, thanks to the extraction $\varphi_2$, two different elements of $(S_{\varphi_4(j)})_{1\leq j\leq r_4}$ are not $*$-connected. Therefore, for $1\leq j<r_4$, we have$$\|x_{\Psi_{in}(j+1)}-x_{\Psi_{out}(j)}\|_1\geq N $$ and so the section $\gamma(j)$ of $\gamma$ from $x_{\Psi_{out}(j)}$ to $x_{\Psi_{in}(j+1)}$ has a diameter larger than $N$ and contains only $p$-open edges. As $B_{out}(j)$ and $B_{in}(j+1)$ are good boxes, we obtain, using again property $(i)$ of good boxes, that $x_{\Psi_{out}(j)}$ and $x_{\Psi_{in}(j+1)}$ belong to the $p$-crossing cluster of their respective boxes. 

Finally, by Lemma \ref{lem2}, for every $j\in\{1,\dots,r_4\}$, there exists a $p$-open path $\gamma_{link}(j)$ joining $x_{\Psi_{in}(j)}$ and $x_{\Psi_{out}(j)}$ of length at most $12\beta N |S_{\varphi_4(j)}|$. We obtain a $p$-open path $\gamma'$ joining $y$ and $z$ by concatenating $\gamma(0),\gamma_{link}(1),\gamma(1),\dots,\gamma_{link}(r_4),$ $\gamma(r_4)$ in this order. Up to removing potential loops, we can suppose that each $\gamma_{link}(j)$ is a self-avoiding path, that all the $\gamma_{link}(j)$ are disjoint and that each $\gamma_{link}(j)$ intersects only $\gamma(j-1)$ and $\gamma(j)$ at their endpoints. Let us estimate the quantity $|\gamma'\setminus\gamma|$, as $\gamma'\setminus\gamma \subset \cup_{i=1}^{r_4} \gamma_{link}(i)$, we obtain:
 \begin{align*}
 |\gamma' \setminus\gamma| &\leq \sum_{j=1}^{r_4} |\gamma_{link}(j)|\\
 & \leq \sum_{j=1}^{r_4} 12\beta N |S_{\varphi_4(j)}| \\
 &\leq 12 \beta N | \gamma_c|+ 12\beta N \sum\limits_{C\in Bad:C\cap\Gamma\neq\emptyset} |\partial_v C| \, 
 \end{align*}
where the last inequality comes from the fact that each $S_{\varphi_4(j)}$ is the union of elements of $ \{\partial_v C:C\in Bad;C\cap\Gamma\neq\emptyset\}$ and of good boxes that contain edges of $\gamma_c$.
We conclude by noticing that $|\partial_v C | \leq 2d |C|$.

\end{proof}
\subsection{Deterministic estimate} When $q-p$ is small, we want to control the probability that the total length of the bypasses $\gamma'\setminus \gamma$ of $p$-closed edges is large. We can notice in Lemma \ref{lem1} that we need to control the bad connected components of the macroscopic site percolation. This will be done in section \ref{ProbaEst}. We will also need a deterministic control on $|\Gamma|$ which is the purpose of the following Lemma (this Lemma is an adaptation of Lemma 3.4 of \cite{GaretMarchandProcacciaTheret}).

\begin{lem}\label{probest} For every path $\gamma$ of $\sZ^d$, for every $N\in\mathbb{N}^*$, there exists a $*$-connected macroscopic path $\widetilde{\Gamma}$ such that 
$$\gamma\subset \bigcup_{\textbf{i}\in\widetilde{\Gamma}}B'_N(\textbf{i}) \text { and }|\widetilde{\Gamma}|\leq 1+\frac{|\gamma|+1}{N}\, .$$
\end{lem}
\begin{proof}
Let $\gamma=(x_i)_{1\leq i \leq n}$ be a path of $\sZ^d$ where $x_i$ is the $i$-th vertex of $\gamma$. Let $\Gamma$ be the set of $N$-boxes that $\gamma$ visits. We are going to define iteratively the macroscopic path $\widetilde{\Gamma}$. Let $p(1)=1$ and $\textbf{i}_1$ be the macroscopic site such that $x_1\in B_N(\textbf{i}_1)$. We suppose that $\textbf{i}_1,\dots,\textbf{i}_k$ and $p(1),\dots, p(k)$ are constructed. 
Let us define $$p(k+1)=\min\left\{j> p(k): x_j \notin B'_N(\textbf{i}_k)\right\}\,.$$ If this set is not empty, we set $\textbf{i}_{k+1}$ to be the macroscopic site such that $$x_{p(k+1)}\in B_N(\textbf{i}_{k+1})\,.$$ Otherwise, we stop the process, and we get that for every $j\in \{p(k),\dots, n\}$, $x_j\in B'_N(\textbf{i}_k)$. As $n$ is finite, the process will eventually stop and the two sequences $(p(1),\dots,p(r))$ and $(\textbf{i}_1,\dots,\textbf{i}_r)$ are finite. Note that the $\textbf{i}_j$ are not necessarily all different. We define $\widetilde{\Gamma}=(\textbf{i}_1,\dots,\textbf{i}_r)$. By construction, 
$$\gamma\subset \bigcup_{\textbf{i}\in\widetilde{\Gamma}}B'_N(\textbf{i})\,.$$
Notice that for every $1\leq k< r$, $\|x_{p(k+1)}-x_{p(k)}\|_1\geq N$, thus $p(k+1)-p(k)\geq N$. This leads to $N(r-1)\leq p(r)-p(1)\leq n$, and finally,
$$|\widetilde{\Gamma}|\leq 1+\frac{|\gamma|+1}{N}\, .$$
\end{proof}
\begin{rk}
This Lemma implies that if $\Gamma$ is the set of $N$-boxes that $\gamma$ visits then 
$$|\Gamma|\leq 3^d |\widetilde{\Gamma}|\leq 3^d \left(1+\frac{|\gamma|+1}{N}\right)\, .$$
\end{rk}
\section{Control of the probability that a box is good}\label{goodbox}

We need in what follows to control the quantity $\sum |C| $ where the sum is over all $C\in Bad$ such that $C\cap\Gamma\neq\emptyset$. We would like to obtain a control which is uniform in the parameter of percolation $p$. To do so, we are going to introduce a parameter $p_0>p_c(d)$ and show that exponential decay is uniform for all $p\geq p_0$. Indeed, the speed will only depend on $p_0$.
\begin{thm}\label{thm5} Let $p_0>p_c(d)$. There exist positive constants $A(p_0)$ and $B(p_0)$ such that for all $p\geq p_0$ and for all $N\geq 1$
$$\Prb(B_N\text{ is $p$-bad})\leq A(p_0)\exp(-B(p_0)N)\,.$$
\end{thm}
\noindent Note that the property $(ii)$ of the definition of $p$-good box is a non-decreasing event in $p$. Thus, it will be easy to bound uniformly the probability that property $(ii)$ is not satisfied by something depending only on $p_0$. However, for properties $(i)$ and $(iii)$ a uniform bound is more delicate to obtain. Before proving Theorem \ref{thm5}, we need the two following lemmas that deal with properties $(i)$ and $(iii)$.
Let $T_{m,N}(p)$ be the event that $B_N$ has a $p$-crossing cluster and contains some other $p$-open cluster $C$ having diameter at least $m$. 
\begin{lem}\label{Grim}
Let $p_0>p_c(d)$, there exist $\nu=\nu(p_0,d)>0$ and $\kappa=\kappa(p_0,d)$ such that for all $p\geq p_0$
\begin{align}\label{tmn}
\Prb(T_{m,N}(p))\leq  \kappa N^{2d}\exp(-\nu m)\, .
\end{align}
\end{lem}
The following Lemma is an improvement of the result of Antal and Pisztora in \cite{AntalPisztora} that controls the probability that two connected points have a too large chemical distance. In the original result, the constants depend on $p$, we slightly modify its proof so that constants are the same for all $p\geq p_0$. This improvement is required to obtain a decay that is uniform in $p$. 
\begin{lem}\label{AP}
Let $p_0>p_c(d)$, there exist $\beta=\beta(p_0)>0$, $\widehat{A}=\widehat{A}(p_0)$ and $\widehat{B}=\widehat{B}(p_0)>0$ such that for all $p\geq p_0$
\begin{align}
\forall x \in\sZ^d\quad\Prb(\beta\|x\|_1\leq D^{\sC'_p}(0,x)<+\infty)\leq \widehat{A}\exp(-\widehat{B}\|x\|_1)\, .
\end{align}
\end{lem}
\begin{rk}
Note that this is not an immediate corollary of \cite{AntalPisztora}. Although increasing the parameter of percolation $p$ reduces the chemical distance, it also increases the probability that two vertices are connected. Therefore the event that we aim to control is neither non-increasing neither non-decreasing in $p$.
\end{rk}
\noindent Before proving these two lemmas, we are first going to prove Theorem \ref{thm5}.
\begin{proof}[Proof of Theorem \ref{thm5}]
Let us fix $p_0>p_c(d)$. Let us denote by $(iii)'$ the property that for all $x,y\in B'_N(\textbf{i})$, if $\|x-y\|_\infty\geq N$ and if $x$ and $y$ belong to the $p$-crossing cluster $\sC$ then $D^{\sC'_p}(x,y)\leq 6\beta N$. Note that properties $(ii)$ and $(iii)'$ imply property $(iii)$. Indeed, thanks to $(ii)$, we can find $z\in\sC\cap B'_N(\textbf{i}) $ such that $\|x-z\|_\infty\geq N$ and $\|y_z\|_\infty\geq N$. Therefore, by applying $(iii)'$,
\begin{align*}
D^{\sC'_p}(x,y)&\leq D^{\sC'_p}(x,z)+D^{\sC'_p}(z,y)\\
&\leq 12\beta N \, .
\end{align*}
Thus, we can bound the probability that a $N$-box is bad by the probability that it does not satisfy one of the properties $(i)$, $(ii)$ or $(iii)'$.
Since we want to control the probability of $B_N$ being a $p$-bad box uniformly in $p$, we will emphasize the dependence of $(i)$, $(ii)$ and $(iii)'$ in $p$ by writing $(i)_p$, $(ii)_p$ and $(iii)'_p$.  First, let us prove that the probability that a $N$-box does not satisfy property $(ii)_p$, i.e., the probability for a box not to have a $p$-crossing cluster, is decaying exponentially, see for instance Theorem 7.68 in \cite{Grimmett99}. There exist positive constants $ \kappa_1(p_0)$ and $ \kappa_2(p_0)$ such that for all $p\geq p_0$
\begin{align}\label{uni2}
\Prb(B_N\text{ does not satisfies }(ii)_p)&\leq\Prb(B_N\text{ does not satisfies }(ii)_{p_0})\nonumber \\
&\leq \kappa_1(p_0)\exp(-\kappa_2(p_0)N^{d-1})\, .
\end{align}
Next, let us bound the probability that a $N$-box does not satisfy property $(iii)'_p$. Using Lemma \ref{AP}, for $p\geq p_0$,
\begin{align*}
\Prb (B_N &\text{ does not satisfy }(iii)'_p)\\
&\leq \sum_{x\in B'_N}\sum_{y\in B'_N} \mathds{1}_{\|x-y\|_\infty\geq N} \Prb\left( 6\beta N\leq D^{\sC'_p}(x,y)<+\infty\right)\\
&\leq \sum_{x\in B'_N}\sum_{y\in B'_N} \mathds{1}_{\|x-y\|_\infty\geq N} \Prb\left( \beta \|x-y\|_\infty\leq D^{\sC'_p}(x,y)<+\infty\right)\\
&\leq \sum_{x\in B'_N}\sum_{y\in B'_N} \mathds{1}_{\|x-y\|_\infty\geq N} \widehat{A}\exp(-\widehat{B}N)\\
&\leq (6N+1)^{2d}\widehat{A}\exp(-\widehat{B}N)\, .
\end{align*}
Finally, by Lemma \ref{Grim},
\begin{align*}
\Prb(B_N&\text{ is $p$-bad})\\
&\leq \Prb(B_N\text{ does not satisfies } (ii)_p)+\Prb(B_N\text{ satisfies $(ii)_p$ but not $(i)_p$})\\
&\hspace{1cm}+\Prb (B_N \text{ does not satisfy }(iii)'_p)\\
&\leq \kappa_1\exp(-\kappa_2N^{d-1}) + 3^d \kappa N^{2d}\exp\left(-\nu\frac{N}{3^d} \right)+(6N+1)^{2d}\widehat{A}\exp(-\widehat{B}N)\\
&\leq A(p_0)e^{-B(p_0)N}\, .
\end{align*}
For the second inequality, we used inequality \eqref{uni2} and the fact that the event that the $3^d$ $N$-boxes of $B'_N$ are crossing and there exist another $p$-open cluster of diameter larger than $N$ in $B'_N$ is included in the event there exists a $N$-box in $B'_N$ that has a crossing property and contains another $p$-open cluster of diameter at least $N/3^d$.
The last inequality holds for $N\geq C_0(p_0)$, where $C_0(p_0)$, $A(p_0)>0$ and $B(p_0)>0$ depends only on $p_0$ and on the dimension $d$.
\end{proof}

\begin{proof}[Proof of Lemma \ref{Grim}]
In dimension $d\geq 3$ , we refer to the proof of Lemma 7.104 in \cite{Grimmett99}. The proof of Lemma 7.104 requires the proof of Lemma 7.78. The probability controlled in Lemma 7.78 is clearly non decreasing in the parameter $p$. Thus, if we choose $\delta(p_0)$ and $L(p_0)$ as in the proof of Lemma 7.78 for $p_0>p_c(d)$, then these parameters can be kept unchanged for some $p\geq p_0$. Thanks to Lemma 7.104, we obtain 
\begin{align*}
\forall p \geq p_0,\, \Prb(T_{m,N}(p))&\leq d(2N+1)^{2d} \exp\left(\left(\frac{m}{L(p_0)+1}-1\right)\log(1-\delta(p_0))\right)\\
&\leq \frac{d.3^d}{1-\delta(p_0)}N^{2d}\exp\left(-\frac{-\log(1-\delta(p_0))}{L(p_0)+1} m\right)\, .
\end{align*}
We get the result with $\kappa= \frac{d.3^d}{1-\delta(p_0)}$ and $\nu=\frac{-\log(1-\delta(p_0))}{L(p_0)+1}>0$.

In dimension 2, the result is obtained by Couronné and Messikh in the more general setting of FK-percolation in Theorem 9 in \cite{COURONNE200481}. We proceed similarly as in dimension $d\geq3$, the constant appearing in this theorem first appeared in Proposition 6. The probability of the event considered in this proposition is clearly increasing in the parameter of the underlying percolation, it is an event for the subcritical regime of the Bernoulli percolation. Let us fix a $p_0>p_c(2)=1/2$, then $1-p_0<p_c(2)$ and we can choose the parameter $c(1-p_0)$ and keep it unchanged for some $1-p\leq 1-p_0$. In Theorem 9, we get the expected result with $c(1-p_0)$ for a $p\geq p_0$ and $g(n)=n$. 
\end{proof}
\noindent We explain now how to modify the proof of \cite{AntalPisztora} to obtain the uniformity in $p$.
\begin{proof}[Proof of Lemma \ref{AP}]
Let $p_0>p_c(d)$ and $p\geq p_0$. First note that the constant $\rho$ appearing in \cite{AntalPisztora} corresponds to our $\beta$. the proof of Lemma 2.3 in \cite{AntalPisztora} can be adapted (as we did above in the proof of Lemma \ref{Grim}) to choose constants $c_3$, $c_4$, $c_6$ and $c_7$ that depend only on $p_0$ and $d$, we do not get into details again. Thanks to this, $N$ may be chosen in the expression $(4.47)$ of \cite{AntalPisztora} such that it only depends on $p_0$ and $d$ and so is $\rho$.  This concludes the proof.
\end{proof}

\section{Probabilistic estimates}\label{ProbaEst}

We can now use the stochastic minoration by a field of independent Bernoulli variables to control the probability that the quantity $\sum |C|$ is big, where the sum is over all $C\in Bad$ such that $C\cap\Gamma\neq\emptyset$. The proof of the following Lemma is in the spirit of the work of Cox and Kesten in \cite{CoxKesten} and relies on combinatorial considerations. These combinatorial considerations were not necessary in \cite{GaretMarchandProcacciaTheret}.

We consider a path $\gamma$ and its associated lattice animal $\Gamma$. We need in the proof of the following Lemma to define $\Gamma$ as a path of macroscopic sites, that is to say a path $(\textbf{i}_k)_{k\leq r}$ in the macroscopic grid such that $\cup_{k\leq r} B_N(\textbf{i}_k)=\Gamma$ (this path may not be self-avoiding). We can choose for instance the sequence of sites that $\gamma$ visits. However, it is difficult to control the size of this sequence by the size of $\Gamma$. That is the reason why we consider the path of the macroscopic grid $\widetilde{\Gamma}$ that was introduced in Lemma \ref{probest}.

\begin{prop}\label{controlBad}
Let $p_0>p_c(d)$ and $\ep\in(0,1-p_c(d))$. There exist a constant $C_\ep\in(0,1)$ depending only on $\ep$ and a positive constant $C_1$ depending on $p_0$, $d$ and $ \beta$, such that if we set $N=C_1|\log\ep|$, then for all $p\geq p_0$ , for every $n\in\sN^*$
$$ \Prb\left(\exists \gamma \text{ starting from $0$ such that}\quad|\widetilde{\Gamma}|\leq n , \quad \sum\limits_{C\in Bad:C\cap\Gamma\neq\emptyset} |C|\geq \ep n\right)\leq C_\ep^ n$$
where $\Gamma$ is the lattice animal associated with the path $\gamma$ and $\widetilde{\Gamma}$ the macroscopic path given by Lemma \ref{probest}.
\end{prop}
\begin{proof}

Let us consider a path $\gamma$ starting from $0$, its associated lattice animal $\Gamma$, \textit{i.e.}, the set of boxes $\gamma$ visits and its associated path on the macroscopic grid  $\widetilde{\Gamma}=(\widetilde{\Gamma}(k))_{0\leq k \leq r}$ as defined in Lemma \ref{probest}.
We first want to include $\widetilde{\Gamma}$ in a subset of the macroscopic grid. Of course, $\widetilde{\Gamma}$ is included in the hypercube of side-length $2r$ centered at $\widetilde{\Gamma}(0)$, but we need to have a more precise control. Let $K\geq 1$ be an integer that we will choose later.
Let $v$ be a site, we denote by $S(v)$ the hypercube of side-length $2K$ centered at $v$ and by $\partial S(v)$ its inner boundary:
  $$S(v)=\{w\in\sZ^d:  \| w-v\|_\infty \leq K\}\quad\text{and}\quad \partial S(v)=\{w\in\sZ^d:  \| w-v\|_\infty =K\}\,.$$ 

\begin{figure}[!ht]
\def\svgwidth{1\textwidth}
\label{fig4}
 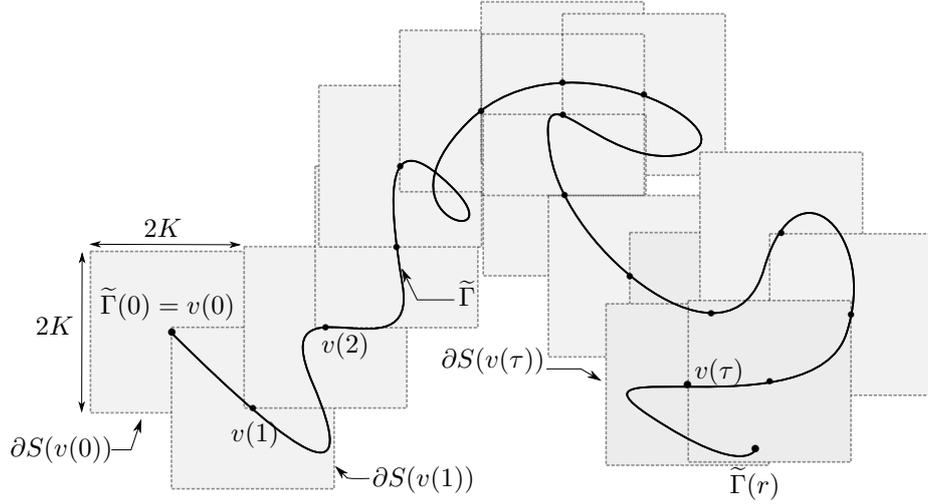
 \caption{Construction of $v(0),\dots,v(\tau)$}
\end{figure}
\noindent We define $v(0)=\widetilde{\Gamma}(0)$, $p_0=0$. If $p_0,\dots,p_k$ and $v(0),\dots,v(k)$ are constructed, we define if any $$p_{k+1}=\min\left\{i\in\{p_{k}+1,\dots,r\}:\widetilde{\Gamma}(i)\in \partial S(v(k))\right\} \quad\text{and} \quad v(k+1)=\widetilde{\Gamma}(p_{k+1})\,.$$ If there is no such index we stop the process.
Since $p_{k+1}-p_{k}\geq K$, there are at most $1+r/K$ such $p_{k}$. Notice that $1+r/K\leq 1+n/K$ on the event $\{|\widetilde{\Gamma}|\leq n\}$. We define $\tau=1+n/K$.
On the event $\{|\widetilde{\Gamma}|\leq n\}$, the macroscopic path $\widetilde{\Gamma}$ is contained in the union of those hypercubes:
$$D(v(0),\dots,v(\tau))=\bigcup_{i=0}^\tau S(v(i))\,.$$
If we stop the process for a $k<\tau$, we artificially complete the sequence until attaining $\tau$ by setting for $k<j\leq \tau$, $v(j)=v(k)$. See figure \ref{fig4}, the corridor $D(v(0),\dots,v(\tau))$ is represented by the grey section. By construction, for all $1\leq k\leq r$, there exists a $j\leq \tau$ such that $\widetilde{\Gamma}(k)$ is in the strict interior of $S(v(j))$, so we have $$\Gamma\subset \bigcup_{k=1}^r\Big\{\,\textbf{j},\, \textbf{j} \text{ is $*$-connected to }\widetilde{\Gamma}(k)\,\Big\}\subset D(v(0),\dots,v(\tau))\,.$$
Thus, we obtain 
\begin{align*}
\Prb&\left(\exists \gamma \text{ starting from $0$ such that}\quad|\widetilde{\Gamma}|\leq n , \quad \sum\limits_{C\in Bad:C\cap\Gamma\neq\emptyset} |C|\geq \ep n\right)\\
&\leq \Prb\left(\bigcup_{v(0),\dots,v(\tau)}\left\{ \begin{array}{c}\exists \gamma \text{ starting from $0$ such that}\\  \sum\limits_{C\in Bad:C\cap\Gamma\neq\emptyset} |C|\geq \ep n,\, \Gamma \subset D(v(0),\dots,v(\tau))\end{array}\right\}\right)\\
&\leq \sum\limits_{v(0),\dots,v(\tau)}\Prb\left(\begin{array}{c}\exists \gamma \text{ starting from $0$ such that}\\ \sum\limits_{C\in Bad:C\cap\Gamma\neq\emptyset} |C|\geq \ep n,\, \Gamma \subset D(v(0),\dots,v(\tau)) \end{array}\right)\\
&\leq \sum\limits_{v(0),\dots,v(\tau)}\Prb\left(  \sum\limits_{\substack{C\in Bad:\\ C\cap D(v(0),\dots,v(\tau))\neq\emptyset}} |C|\geq \ep n \right)\\
&\leq\sum\limits_{v(0),\dots,v(\tau)}\sum_{j\geq \ep n}\Prb\left(  \sum\limits_{\substack{C\in Bad:\\ C\cap D(v(0),\dots,v(\tau))\neq\emptyset}} |C|=j\right)
\end{align*}
where the first sum is over the sites $v(0),\dots,v(\tau)$ satisfying $v(0)=\Gamma(1)$ and for all $0\leq k <\tau$, $v(k+1)\in \partial S(v(k))\cup \{v(k)\}$. Since $\partial S(v)\cup\{v\}$ contains at most $(c_dK)^{d-1}$ sites where $c_d\geq 1$ is a constant depending only on the dimension, the sum over the sites $v(0),\dots,v(\tau)$ contains at most $$(c_dK)^{(d-1)\tau}\leq(c_dK)^{\frac{2n(d-1)}{K}}:=C_2^n$$ terms for $n$ large enough. For any fixed $v(0),\dots, v(\tau)$, $D(v(0),\dots,v(\tau))$ contains at most $$(\tau+1)(2K+1)^d\leq(n/K+2)(2K+1)^d\leq 2 n (3K)^d:=C_3n$$ macroscopic sites. Let us recall that for a bad macroscopic site $\textbf{i}$,  $C(\textbf{i})$ denotes the connected cluster of bad macroscopic sites containing $\textbf{i}$. Let us notice that the following event $$\left\{ \sum\limits_{\substack{C\in Bad:\\ C\cap D(v(0),\dots,v(\tau))\neq\emptyset}} |C|=j\right\}$$ is included in the event:  there exist an integer $\rho\leq C_3n$ and distinct bad macroscopic sites $\textbf{i}_1,\dots,\textbf{i}_\rho\in D(v(0),\dots,v(\tau)) $, disjoint connected components $\bar{C}_1,\dots,\bar{C}_\rho$ such that for all $1\leq k \leq \rho$, $C(\textbf{i}_k)=\bar{C}_k$ and $\sum_{k=1}^\rho|\bar{C}_k|=j $.
Therefore, for any fixed $v(0),\dots, v(\tau)$, 
\begin{align}\label{bar1}
&\Prb\left(  \sum\limits_{\substack{C\in Bad:\\ C\cap D(v(0),\dots,v(\tau))\neq\emptyset}} |C|=j\right)\nonumber\\
&=\sum_{\rho=1}^{C_3n}\,\sum_{\substack{\textbf{i}_1\in D(v(0),\dots,v(\tau))\\\dots\\ \textbf{i}_\rho\in D(v(0),\dots,v(\tau))\\ \forall k\neq l,\textbf{i}_k\neq \textbf{i}_l}}\,\sum_{\substack{j_1,\dots,j_\rho\geq 1\\ j_1+\dots+j_\rho=j}} \,\sum_{\substack{C_1\in \Animals_{\textbf{i}_1}^{j_1}\\ \dots \\ C_\rho\in \Animals_{\textbf{i}_\rho}^{j_\rho}}}\Prb\left(\begin{array}{c}\forall 1\leq k \leq \rho\\C(\textbf{i}_k)=\bar{C}_k,\\ \,\sum_{k=1}^\rho|\bar{C}_k|=j\end{array}\right)
\end{align}  
where $\Animals_\textbf{v}^k$ is the set of connected macroscopic sites of size $k$ containing the site $\textbf{v}$. We have $|\Animals_\textbf{v}^k|\leq(7^d)^k$ (see for instance Grimmett \cite{Grimmett99}, p85). There are at most $\binom{C_3n}{\rho}$ ways of choosing the sites $\textbf{i}_1,\dots,\textbf{i}_\rho$. Thus, if we fix the sites  $\textbf{i}_1,\dots,\textbf{i}_\rho$ the number of possible choices of the connected components $\bar{C}_1,\dots,\bar{C}_\rho$ such that for all $1\leq k \leq \rho$, $C(\textbf{i}_k)=\bar{C}_k$ and $\sum_{k=1}^\rho|\bar{C}_k|=j$ is at most: 
$$\sum_{\substack{j_1,\dots,j_\rho\geq 1\\ j_1+\dots+j_\rho=j}} (7^d)^{j_1}\cdots(7^d)^{j_\rho}=(7^d)^j \sum_{\substack{j_1,\dots,j_\rho\geq 1\\ j_1+\dots+j_\rho=j}} 1\, .$$
Next we need to estimate, for given sites  $\textbf{i}_1,\dots,\textbf{i}_\rho$  and disjoint connected components  $\bar{C}_1,\dots,\bar{C}_\rho$, the probability  that for all $1\leq k \leq \rho$, $C(\textbf{i}_k)=\bar{C}_k$. For all sites $\textbf{i}\in\cup _{k=1}^\rho \bar{C}_k$, the $N$-box $B_N(\textbf{i})$ is bad. There is a short range of dependence between the state of the boxes. However, by definition of a $p$-good box, the state of $B_N(\textbf{i})$ only depends on boxes $B_N(\textbf{j})$ such that $\|\textbf{i}-\textbf{j}\|_\infty \leq 13\beta$. Thus, if  $\|\textbf{i}-\textbf{j}\|_\infty \geq 27\beta$ the state of the boxes $B_N(\textbf{i})$ and $B_N(\textbf{j})$ are independent. We can deterministically extract from $\cup _{k=1}^\rho \bar{C}_k$ a set of macroscopic site $\cE$ such that $|\cE|\geq j/(27\beta)^d$  and for any $\textbf{i}\neq\textbf{j} \in\cE$, the state of the boxes $B_N(\textbf{i})$ and $B_N(\textbf{j})$ are independent. Therefore, we have using Proposition \ref{thm5}
\begin{align}\label{bar2}
\Prb\left(\forall 1\leq k \leq \rho,\,C(\textbf{i}_k)=\bar{C}_k, \,\sum_{k=1}^\rho|\bar{C}_k|=j\right)&\leq \Prb \left(\forall \textbf{i}\in\cE, \,B_N(\textbf{i})\text{ is $p$-bad}\right)\nonumber\\
&\leq \Prb(B_N(\textbf{0})\text{ is $p$-bad})^{ j/(27\beta)^d}\nonumber\\
&\leq \left(A(p_0)\exp(-B(p_0)N(\ep))\right) ^{ j/(27\beta)^d}\,.
\end{align}
In what follows, we set $\alpha=\alpha(\ep)=  \left(A(p_0)\exp(-B(p_0)N(\ep))\right)^{ 1/(27\beta)^d}$ in order to lighten the notations. We aim to find an expression of $\alpha(\ep)$ such that we get the upper bound stated in the Proposition. The expression of $N(\ep)$ will be determined by the choice of $\alpha(\ep)$.
Combining inequalities \eqref{bar1} and \eqref{bar2}, we obtain
\begin{align*}
\Prb\left(  \sum\limits_{\substack{C\in Bad:\\ C\cap D(v(0),\dots,v(\tau))\neq\emptyset}} |C|=j\right)\leq\binom{C_3n}{\rho} (7^d\alpha)^j \sum_{\substack{j_1,\dots,j_\rho\geq 1\\ j_1+\dots+j_\rho=j}} 1\, 
\end{align*}
and so
\begin{align*}
\Prb&\left(\exists \gamma \text{ starting from $0$ such that}\quad|\widetilde{\Gamma}|\leq n , \quad \sum\limits_{C\in Bad:C\cap\Gamma\neq\emptyset} |C|\geq \ep n\right)\\
&\hspace{4cm}\leq C_2^n\sum_{j\geq\ep n}   (7^d\alpha)^j \sum_{\rho=1}^{C_3n}\binom{C_3n}{\rho}\hfill \sum_{\substack{j_1,\dots,j_\rho\geq 1\\ j_1+\dots+j_\rho=j}} 1\, .
\end{align*}
Notice that
\begin{align*}
\sum_{\rho=1}^{C_3n}\binom{C_3n}{\rho} \sum_{\substack{j_1,\dots,j_\rho\geq 1\\ j_1+\dots+j_\rho=j}} 1= \sum_{\substack{j_1,\dots, j_{C_3.n}\geq 0\\ j_1+\dots+ j_{C_3.n}=j}}  1 = \binom{C_3 n+j-1}{j}\,.
\end{align*}
To bound those terms we will need the following inequality, for $r\geq 3$, $N\in\sN^*$ and a real $z$ such that  $0<ez(1+\frac{r}{N})<1$:
\begin{align}\label{stirling}
\sum_{j=N}^\infty z^j\binom{r+j-1}{j}\leq \nu \frac{(ez(1+\frac{r}{N}))^N}{1-ez(1+\frac{r}{N})}
\end{align} 
where $\nu$ is an absolute constant. This inequality was present in \cite{CoxKesten} but without proof, for completeness we will give a proof of \eqref{stirling} at the end of the proof of Proposition \ref{controlBad}. Using inequality \eqref{stirling} and assuming $0<e7^d\alpha(\ep) (1+\frac{C_3}{\ep })<1$, we get,
\begin{align*}
\Prb&\left(\exists \gamma \text{ starting from $0$ such that}\quad|\widetilde{\Gamma}|\leq n , \quad \sum\limits_{C\in Bad:C\cap\Gamma\neq\emptyset} |C|\geq \ep n\right)\\
&\hspace{5cm}\leq C_2^n\sum_{j\geq\ep n}   (7^d\alpha)^j \binom{C_3 n+j-1}{j}\\
&\hspace{5cm}\leq  \nu C_2^n\frac{\left[e7^d\alpha (\ep)(1+\frac{C_3}{\ep })\right]^{\ep n}}{1-e7^d\alpha(\ep) (1+\frac{C_3}{\ep })} \, .
\end{align*}
Let us recall that $C_2=(c_dK)^{2(d-1)/K}$ and $C_3=2(3K)^d$. We have to choose $K(\ep)$, $\alpha(\ep)$ and a constant $0<C_\ep<1$ such that $C_2\left[e7^d\alpha (\ep)(1+\frac{C_3}{\ep })\right]^{\ep }<C_\ep$ that is to say 
\begin{align}\label{cond2}
(c_dK)^\frac{2(d-1)}{K}\left[e7^d\alpha (\ep)(1+\frac{2(3K)^d}{\ep })\right]^{\ep }<C_\ep\, .
\end{align} 
Note that the condition \eqref{cond2} implies the condition $0<e7^d\alpha(\ep) (1+\frac{C_3}{\ep })<1$. We fix $K$ the unique integer such that $\frac{1}{\ep}\leq K<\frac{1}{\ep}+1\leq \frac{2}{\ep}$. We recall that $\ep<1$. Thus,
\begin{align*}
(c_dK)^\frac{2(d-1)}{K}&\left[e7^d\alpha (\ep)(1+\frac{2(3K)^d}{\ep})\right]^{\ep }\\
&\leq (c_dK)^\frac{2d}{K}\left[e7^d\alpha (\ep)\frac{4(3K)^d}{\ep}\right]^{\ep}\\
&\leq \exp\left[\frac{2d}{K}\log (c_dK)+\ep \log \left(e7^d\alpha (\ep)\frac{4(3K)^d}{\ep }\right) \right]\\
&\leq \exp\left[2d\ep\log\left(\frac{2c_d}{\ep}\right)+\ep \log \left(e7^d\alpha (\ep)\frac{4(3\frac{2}{\ep})^d}{\ep }\right) \right]\\
&\leq \exp\Bigg[-2d\ep\log\ep+d\ep\log(2c_d)
 +\ep \log \left(4e(42)^d\alpha (\ep)\frac{1}{\ep^{d+1}}\right) \Bigg]\, .\\
\end{align*}
We set $$\alpha(\ep)=\left(2c_d\right)^{d}\frac{\ep^r}{4e(42)^d}$$ where $r$ is the smallest integer such that $r\geq 3d+2$. We obtain 
\begin{align*}
(c_dK)^\frac{d}{K}\left[e7^d\alpha (\ep)(1+\frac{2(3K)^d}{\ep})\right]^{\ep }&\leq \exp((r-(3d+1))\ep\log\ep)\\
&\leq \exp(\ep\log\ep)<1\,.
\end{align*}
Therefore there exists a positive constant $C_1$ depending on $\beta$, $d$, $p_0$ such that 
$$N(\ep)=C_1|\log \ep|\,.$$
It remains now to prove inequality \eqref{stirling} to conclude. To show this inequality, we need a version of Stirling's formula with bounds: for all $n\in\sN^*$, one has
$$\sqrt{2\pi}\,n^{n+\frac{1}{2}}e^{-n}\leq n! \leq e\,n^{n+\frac{1}{2}}e^{-n}\,,$$
thus,
\begin{align*}
\sum_{j=N}^\infty z^j\binom{r+j-1}{j}&=\sum_{j=N}^\infty z^j\frac{(r+j-1)!}{j!(r-1)!}\\
&\leq \sum_{j=N}^\infty z^j \frac{e\, (r+j-1)^{r+j-\frac{1}{2}}e^{-(r+j-1)}}{2\pi \,j^{j+\frac{1}{2}}(r-1)^{r-\frac{1}{2}}e^{-(r+j-1)}}\\
&= \sum_{j=N}^\infty \frac{e}{2\pi} z^j \left(\frac{r+j-1}{j}\right)^j \left(\frac{r+j-1}{r-1}\right)^{r-\frac{1}{2}}j^{-\frac{1}{2}}\\
&\leq \sum_{j=N}^\infty \frac{e}{2\pi} z^j \left(1+\frac{r}{N}\right)^j \left(1+\frac{j}{r-1}\right)^{r-1}\left(\frac{1}{j}+\frac{1}{r-1}\right)^{\frac{1}{2}}\\
&\leq \sum_{j=N}^\infty \frac{e}{2\pi} z^j \left(1+\frac{r}{N}\right)^j e^{(r-1)\log(1+j/(r-1))}\\
&\leq \sum_{j=N}^\infty \frac{e}{2\pi} (ez)^j \left(1+\frac{r}{N}\right)^j=  \frac{e}{2\pi}  \frac{(ez(1+\frac{r}{N}))^N}{1-ez(1+\frac{r}{N})}
\end{align*}
where we use in the last inequality the fact that for all $x>0$, $\log(1+x)\leq x$.
\end{proof}

\section{Regularity of the time constant}\label{estimate}
In this section, we prove the main result Theorem \ref{heart} and its Corollary \ref{cor1}. Before proving this Theorem, we need to prove two lemmas.
The following Lemma  enables to control the number of $p$-closed edges $|\gamma_c|$ in a geodesic $\gamma$ between two given points $y$ and $z$ in the infinite cluster $\sC_p$. 
We denote by $F_x$ the event that $0,x\in\sC_p$ and the $N$-boxes containing $0$ and $x$ belong to an infinite cluster of $p$-good boxes.
\begin{lem}\label{geo}
Let $p_c(d)< p\leq q$. Let us consider $x\in\sZ^d$. Then, for  $\delta>0$
\begin{align*}
\Prb\left(F_x,\,D^{\sC_p}(0,x)> D^{\sC_q}(0,x)\left(1+\rho_d N\left(\frac{q-p}{q}+\delta\right)\right)+\rho_dN \sum\limits_{\substack{C\in Bad:\\C\cap\Gamma\neq\emptyset}} |C| \right)&\\
\leq e^{-2\delta^2{ \|x\|_1}}\, .&
\end{align*}
where $\Gamma$ is the lattice animal of $N$-boxes visited by an optimal path $\gamma$ between $0$ and $x$ in $\sC_q$.

\end{lem}

\begin{proof} On the event $F_x$, we have $0,x\in\sC_p\subset\sC_q$ so there exists a $q$-open path joining $0$ to $x$, let $\gamma$ be an optimal one. Necessarily, we have $|\gamma|\geq \|x\|_1$. We consider the modification $\gamma'$ given by Lemma \ref{lem1}. As $\gamma'$ is $p$-open,
\begin{align}\label{eq6.1.1}
D^{\sC_p}(0,x)<|\gamma'|&\leq |\gamma\cap\gamma'|+|\gamma'\setminus\gamma|\nonumber\\
&\leq |\gamma|+\rho_d\left(N|\gamma_c|+N\sum\limits_{C\in Bad:C\cap\Gamma\neq\emptyset} |C| \right)\nonumber\\
&\leq D^{\sC_q}(0,x)+\rho_d\left(N|\gamma_c|+N\sum\limits_{C\in Bad:C\cap\Gamma\neq\emptyset} |C| \right)\, .
\end{align}
We want to control the size of $\gamma_c$. For that purpose, we want to introduce a coupling of the percolations $q$ and $p$, such that if any edge is $p$-open then it is $q$-open, and we want the random path $\gamma$, which is an optimal $q$-open path between $0$ and $x$, to be independent of the $p$-state of any edge, i.e., any edge is $p$-open or $p$-closed independently of $\gamma$. This is not the case when we use the classic coupling with a unique uniform random variable for each edge. Here we introduce two sources of randomness to ease the computations by making the choice of $\gamma$ independent from the $p$-state of its edges. We proceed in the following way: with each edge we associate two independent Bernoulli random variables $V$ and $Z$ of parameters respectively $q$ and $p/q$. Then $W=Z\cdot V$ is also a Bernoulli random variable of parameter $p$. This implies
\begin{align*}
\Prb(W=0|V=1)&=\Prb(Z=0|V=1)=\Prb(Z=0)=1-\frac{p}{q}=\frac{q-p}{q}\,.
\end{align*}
Thus, we can now bound the following quantity by summing on all possible self-avoiding paths for $\gamma$. For short, we use the abbreviation s.a. for self-avoiding.
\begin{align}\label{couplage1}
\Prb\Bigg(|\gamma_c|&\geq |\gamma|\left(\frac{q-p}{q}+\delta\right)\Bigg)\nonumber\\
&=\sum_{k=\|x\|_1}^\infty \sum_{\substack{|r|=k\\ \text {r s.a. path }}} \Prb\left(\gamma=r,|\gamma_c|\geq |\gamma|\left(\frac{q-p}{q}+\delta\right)\right)\nonumber\\
&=\sum_{k=\|x\|_1}^\infty \sum_{\substack{|r|=k\\ \text {r s.a. path }}} \Prb\left(\gamma=r,|\{e\in r:\,e\text{ is $p$-closed}\}|\geq k\left(\frac{q-p}{q}+\delta\right)\right)\nonumber\\
&=\sum_{k=\|x\|_1}^\infty \sum_{\substack{|r|=k\\ \text {r s.a. path }}} \Prb\left(\gamma=r, |\{e\in r: Z(e)=0\}| \geq k\left(\frac{q-p}{q}+\delta\right)\right)\nonumber\\
&=\sum_{k=\|x\|_1}^\infty \sum_{\substack{|r|=k\\ \text {r s.a. path }}} \Prb\left(\gamma=r\right)\Prb\left( |\{e\in r: Z(e)=0\}| \geq k\left(\frac{q-p}{q}+\delta\right)\right)\nonumber\\
&\leq\sum_{k=\|x\|_1}^\infty \sum_{\substack{|r|=k\\ \text {r s.a. path }}}  \Prb\left(\gamma=r\right) e^{-2\delta^2 k}\leq e^{-2\delta^2 \|x\|_1}
\end{align}
where we use Chernoff bound in the second to last inequality (see Theorem 1 in \cite{Chernoff}).
On the event $F_x\cap\left\{|\gamma_c|< |\gamma|\left(\frac{q-p}{q}+\delta\right)\right\}$, by \eqref{eq6.1.1}, we get
\begin{align*}
D^{\sC_p}(0,x)&\leq D^{\sC_q}(0,x)+\rho_d\left(N |\gamma|\left(\frac{q-p}{q}+\delta\right)+N\sum\limits_{C\in Bad:C\cap\Gamma\neq\emptyset} |C| \right)\\
&= D^{\sC_q}(0,x)\left(1+\rho_dN\left(\frac{q-p}{q}+\delta\right)\right)+\rho_dN\sum\limits_{C\in Bad:C\cap\Gamma\neq\emptyset} |C|
\end{align*}
and the conclusion follows.
\end{proof}
\noindent The proof of the following Lemma is the last step before proving Theorem \ref{heart}.
\begin{lem}\label{controle}
Let $p_0>p_c(d)$ and $\ep\in(0,1-p_0)$, we set $N(\ep)$ as in Proposition \ref{controlBad}. There exists $\pp:=\pp(\ep,p_0)>0$ such that for all $q\geq p\geq p_0$, for all $x\in\sZ^d$ with $\|x\|_1$ large enough,
\begin{align*}
\Prb\left(D^{\sC_p}(\widetilde{0}^{\sC_p},\widetilde{x}^{\sC_p})\leq  D^{\sC_q}(\widetilde{0}^{\sC_p},\widetilde{x}^{\sC_p})\left(1+\rho_d\frac{q-p}{q}N(\ep)\right)+\eta_d \ep\|x\|_1\right)\geq \pp(\ep,p_0)
\end{align*}
where $\eta_d>0$ is a constant depending only on $d$.
\end{lem}
\begin{proof}
Let us fix $\ep>0$ and $N(\ep)$ as in Proposition 5.1.  Fix an $x\in\sZ^d$ such that $\|x\|_1\geq 3dN(\ep)$. We denote by $B_{N(\ep)}(0)$ (respectively $B_{N(\ep)}(x)$) the $N(\ep)$-box containing $0$ (rep. $x$) and by $\underline{\sC}_p$ the union of infinite cluster of $p$-good boxes. We recall that 
$$F_x=\big\{\,0\in\sC_p,\,x\in\sC_p\,\big\}\cap\big\{\,B_{N(\ep)}(0)\in\underline{\sC}_p,\,B_{N(\ep)}(x)\in\underline{\sC}_p\,\big\}\,.$$
We have
\begin{align}\label{iq1}
\Prb&\left(D^{\sC_p}(\widetilde{0}^{\sC_p},\widetilde{x}^{\sC_p})\geq  D^{\sC_q}(\widetilde{0}^{\sC_p},\widetilde{x}^{\sC_p})\left(1+\rho_d\frac{q-p}{q}N(\ep)\right)+3\ep\beta\rho_d\|x\|_1\right)\nonumber\\
&\leq \Prb\left(F_x,\, D^{\sC_p}(0,x)\geq  D^{\sC_q}(0,x)\left(1+\rho_d\frac{q-p}{q}N(\ep)\right)+3\ep\beta\rho_d\|x\|_1\right)+\Prb(F_x^c)\,.
\end{align}
We have $$\Prb(F_x^c)\leq \Prb (\{0\in\sC_p,\,x\in\sC_p\}^c)+\Prb(\{B_{N(\ep)}(0)\in \underline{\sC}_p,\,B_{N(\ep)}(x)\in \underline{\sC}_p\}^c)\,.$$
Using FKG inequality, we have
$$\Prb(0\in\sC_p,\,x\in\sC_p)\geq \Prb(0\in\sC_p)\Prb(x\in\sC_p)\geq \theta_{p_0}^2\,.$$ 
Let us define $Y_\textbf{i}=\mathds{1}_{\{B_{N(\ep)}(\textbf{i}) \text{ is $p$-good}\}}$. First note that the field $(Y_\textbf{i})_{\textbf{i}\in\sZ^d}$ has a finite range of dependence that depends on $\beta$ and $d$. Using the stochastic comparison in \cite{Liggett}, for every $\pp_1$, there exists a positive constant $\alpha$ depending on $\beta$, $d$ and $\pp_1$ such that if $\Prb(Y_\textbf{0}=0)\leq \alpha$ then the field $(Y_{\textbf{i}}) _{\textbf{i}\in\sZ^d}$ stochastically dominates a family of independent Bernoulli random variables with parameter $\pp_1$.
Let us choose $\pp_1$ large enough such that $$1-\theta_{site,\pp_1}^2\leq \frac{\theta_{p_0} ^2}{2}\,$$
where $\theta_{site,\pp_1}$ denotes the probability for a site to belong to the infinite cluster of i.i.d. Bernoulli site percolation of parameter $\pp_1$. Thanks to Theorem \ref{thm5}, there exists a positive integer $N_0$ depending only on $\alpha$, $p_0$ and $d$ such that for every $N\geq N_0$, 
$$\Prb(Y_\textbf{0}=0)\leq \alpha\,.$$ 
For every $\ep\leq 1-p_0$, we have $|\log \ep|\geq |\log (1-p_0)|$. Up to taking a larger constant $C_1$ in the expression of $N(\ep)$ stated in Proposition \ref{controlBad}, \textit{i.e.}, $N(\ep)=C_1|\log\ep|$, we can assume without loss of generality that $N(\ep)\geq N_0$ so that  using the stochastic domination and FKG we obtain
$$\Prb(B_{N(\ep)}(0)\in \underline{\sC}_p,\,B_{N(\ep)}(x)\in \underline{\sC}_p)\geq \theta_{site,\pp_1}^2\,.$$
Finally, we get
\begin{align}\label{iq1'}
\Prb(F_x^c)&\leq 1-\theta_{p_0} ^2+ 1-\theta_{site,\pp_1}^2\leq 1- \frac{\theta_{p_0} ^2}{2}\,.
\end{align}

On the event $F_x$, we have $0,x\in\sC_p\subset\sC_q$, we can consider $\gamma$ a geodesic from $0$ to $x$ in $\sC_q$, and let $\Gamma$ be the set of $N$-boxes that $\gamma$ visits. 

\noindent By Lemma \ref{geo}, we have for every $\delta>0$
\begin{align}\label{iq2}
\Prb&\Big(F_x,\, D^{\sC_p}(0,x)\geq  D^{\sC_q}(0,x)\left(1+\rho_d\frac{q-p}{q}N(\ep)\right)+3\ep\beta\rho_d\|x\|_1\Big)\nonumber\\
\leq &\Prb\left(F_x, \,\rho_dN(\ep)\left(D^{\sC_q}(0,x)\delta +\sum\limits_{C\in Bad:C\cap\Gamma\neq\emptyset} |C|\right)\geq 3\ep\beta\|x\|_1\right)\nonumber\\
&+\Prb\left(\begin{array}{c}F_x,\, D^{\sC_p}(0,x)> D^{\sC_q}(0,x)\left(1+\rho_d N(\ep)\left(\frac{q-p}{q}+\delta\right)\right)\\\hfill+\rho_dN(\ep) \sum\limits_{\substack{C\in Bad:\\C\cap\Gamma\neq\emptyset}} |C|\end{array}\right)\nonumber\\
\leq &\Prb\left(F_x,\,|\gamma|\leq\beta\|x\|_1,\,\sum\limits_{C\in Bad:C\cap\Gamma\neq\emptyset} |C|\geq \frac{3\ep\beta\|x\|_1}{N(\ep)}-\delta|\gamma|\right)\nonumber\\
&+ \Prb\left(F_x,\,|\gamma|>\beta\|x\|_1\right)+e^{-2\delta^2\|x\|_1}\nonumber\\
\leq &\Prb\left(F_x,\,|\gamma|\leq\beta\|x\|_1,\,\sum\limits_{C\in Bad:C\cap\Gamma\neq\emptyset} |C|\geq \beta\|x\|_1 \left(\frac{3\ep}{N(\ep)}-\delta\right)\right)\nonumber\\
&+\Prb\left(F_x,\,|\gamma|>\beta\|x\|_1\right)+e^{-2\delta^2\|x\|_1}\, .
\end{align}
We set $\delta=\ep/N(\ep)$. We know by Lemma \ref{probest} that $|\widetilde{\Gamma}|\leq 1+(|\gamma|+1)/N(\ep)$. Moreover as $|\gamma|\geq 3dN(\ep)$, we have $|\widetilde{\Gamma}|\leq 2|\gamma|/N(\ep)$. Using Proposition \ref{controlBad},
\begin{align}\label{iq3}
\Prb&\left(F_x,\,|\gamma|\leq\beta\|x\|_1,\,\sum\limits_{C\in Bad:C\cap\Gamma\neq\emptyset} |C|\geq\beta\|x\|_1 \left(\frac{3\ep}{N(\ep)}-\delta\right)\right) \nonumber\\
&\leq\Prb\left(\begin{array}{c}\exists \gamma \text{ starting from $0$ such that}\, |\widetilde{\Gamma}|\leq\frac{ 2\beta\|x\|_1}{N(\ep)},\\\sum\limits_{C\in Bad:C\cap\Gamma\neq\emptyset} |C|\geq\ep\frac{ 2\beta\|x\|_1}{N(\ep)}\end{array}\right)\leq C_\ep^{2\beta\|x\|_1/N(\ep)}
\end{align}
where $C_\ep<1$.
Moreover, by Lemma \ref{AP}, we get 
\begin{align}\label{iq4}
\Prb\left(F_x,\,|\gamma|>\beta\|x\|_1\right)&\leq  \Prb(\beta\|x\|_1\leq D^{\sC_q}(0,x)<+\infty)\leq \widehat{A}\exp(-\widehat{B}\|x\|_1)\, .
\end{align}
Finally, combining \eqref{iq1}, \eqref{iq1'}, \eqref{iq2}, \eqref{iq3} and  \eqref{iq4}, we obtain that
\begin{align*}
\Prb&\left(D^{\sC_p}(\widetilde{0}^{\sC_p},\widetilde{x}^{\sC_p})\geq  D^{\sC_q}(\widetilde{0}^{\sC_p},\widetilde{x}^{\sC_p})\left(1+\rho_d\frac{q-p}{q}N(\ep)\right)+3\ep\beta\rho_d\|x\|_1\right)\\
&\leq 1- \frac{\theta_{p_0}^2}{2}+ C_\ep^{2\beta\|x\|_1/N(\ep)}+\widehat{A}e^{-\widehat{B}\|x\|_1}+e^{-2\ep^2\|x\|_1/N(\ep)^2} \\
&\leq 1-\pp(\ep,p_0)
\end{align*}
for an appropriate choice of $\pp(\ep,p_0)>0$ and for every $x$ such that $\|x\|_1$ is large enough.
\end{proof}

\begin{proof}[Proof of Theorem \ref{heart}]
Let $\ep>0$, $\delta>0$, $p_0>p_c(d)$ and $x\in\sZ^d$, consider $N(\ep)=C_1|\log\ep|$ as in Proposition \ref{controlBad}, $\pp=\pp(\ep,p_0)$ as in Lemma \ref{controle} and $q\geq p\geq p_0$. With the convergence of the regularized times given by Proposition \ref{convergence}, we can choose $n$ large enough such that 
$$\Prb\left(\mu_p(x)-\delta \leq \frac{D^{\sC_p}(\widetilde{0}^{\sC_p},\widetilde{nx}^{\sC_p})}{n}\right)\geq 1-\frac{\pp}{3}$$
$$\Prb\left(  \frac{D^{\sC_q}(\widetilde{0}^{\sC_p},\widetilde{nx}^{\sC_p})}{n}\leq\mu_q(x)+\delta\right)\geq 1-\frac{\pp}{3}$$
$$\Prb\left(D^{\sC_p}(\widetilde{0}^{\sC_p},\widetilde{nx}^{\sC_p})\leq  D^{\sC_q}(\widetilde{0}^{\sC_p},\widetilde{nx}^{\sC_p})\left(1+\rho_d\frac{q-p}{q}N(\ep)\right)+\eta_d\ep n\|x\|_1\right)\geq \pp\,.$$
The intersection of these three events has positive probability, we obtain on this intersection
\begin{align*}
\mu_p(x)-\delta\leq (\mu_q(x)+\delta)\left(1+\rho_d\frac{q-p}{q}N(\ep)\right)+\eta_d\ep \|x\|_1\,.
\end{align*}
By taking the limit when $\delta$ goes to $0$ we get
\begin{align*}
\mu_p(x)\leq \mu_q(x)\left(1+\rho_d\frac{q-p}{q}N(\ep)\right)+\eta_d\ep \|x\|_1\,. 
\end{align*}
By Corollary \ref{decmu}, we know that the map $p\rightarrow\mu_p$ is non-increasing. We also know that $\mu_p(x)\leq \|x\|_1 \mu_p(e_1)$ for $e_1=(1,0,\dots,0)$, for any $p>p_c(d)$ and any $x\in\sZ^d$. Thus, for every $\ep>0$,
\begin{align*}
\mu_p(x)- \mu_q(x)&\leq \mu_q(x)\rho_d\frac{q-p}{q}N(\ep)+\eta_d\ep \|x\|_1\\
&\leq \mu_{p_0}(e_1)\|x\|_1\rho_d\frac{q-p}{p_c(d)}N(\ep)+\eta_d\ep \|x\|_1\\
&\leq \eta'_d (p_0)\|x\|_1(N(\ep)(q-p)+\ep)
\end{align*}
where $\eta'_d(p_0)$ is a constant depending on $d$ and $p_0$.
Using the expression of $N(\ep)$ stated in Proposition \ref{controlBad}, we obtain
\begin{align}\label{finaleq}
\mu_p(x)- \mu_q(x)&\leq  \eta'_d \|x\|_1\left(C_1|\log\ep|(q-p)+\ep\right)\, .
\end{align}
By setting $\ep=q-p$ in the inequality, we get
\begin{align*}
\mu_p(x)- \mu_q(x)&\leq  \eta''_d \|x\|_1(q-p)|\log(q-p)|
\end{align*}
where $\eta''_d>0$ depends only on $p_0$ and $d$. 
Thanks to Corollary \ref{decmu}, we have $\mu_p(x)- \mu_q(x)\geq 0$, so that
\begin{align}\label{eq6.1}
|\mu_p(x)- \mu_q(x)|&\leq  \eta''_d \|x\|_1(q-p)|\log(q-p)|\, .
\end{align}
By homogeneity, \eqref{eq6.1} also holds for all $x\in\mathbb{Q}^d$. Let us recall that for all $x,y\in\mathbb{R}^d$ and $p\geq p_c(d)$, 
\begin{align}\label{eqcerf}
|\mu_p(x)- \mu_p(y)|\leq \mu_p(e_1)\|x-y\|_1\,,
\end{align}
see for instance Theorem 1 in \cite{cerf2016}.
Moreover, there exists a finite set $(y_1,\dots,y_m)$ of rational points of $\mathbb{S}^{d-1}$ such that 
\begin{align*}
\mathbb{S}^{d-1}\subset \bigcup_{i=1}^m\Big\{\,x\in\mathbb{S}^{d-1}:\|y_i-x\|_1\leq (q-p)|\log(q-p)|\,\Big\}\, .
\end{align*}
Let $x\in\mathbb{S}^{d-1}$ and $y_i$ such that $\|y_i-x\|_1\leq (q-p)|\log(q-p)|$. Using inequality \eqref{eqcerf}, we get
\begin{align*}
|\mu_p(x)&- \mu_q(x)|\\
&\leq |\mu_p(x)- \mu_p(y_i)|+|\mu_p(y_i)- \mu_q(y_i)|+|\mu_q(y_i)- \mu_q(x)|\\
&\leq \mu_p(e_1)\|y_i-x\|_1+    \eta''_d \|y_i\|_1(q-p)|\log(q-p)|  + \mu_q(e_1)\|y_i-x\|_1\\
&\leq \left(2\mu_{p_0}(e_1)+\eta''_d \right)(q-p)|\log(q-p)|\,.
\end{align*}
This yields the result.
\end{proof}
\begin{proof}[Proof of Corollary \ref{cor1}]
Let $p_0>p_c(d)$. We consider the constant $\kappa_d$ appearing in the Theorem \ref{heart}. Let $p\leq q$ in $[p_0,1]$.
We recall the following definition of the Hausdorff distance between two subsets $E$ and $F$ of $\sR^d$:
$$d_\cH(E,F)=\inf \Big\{\,r\in\sR^+:E\subset F^r\text{ and } F\subset E^r\,\Big\}$$
where $E^r=\{y:\exists x\in E, \|y-x\|_2\leq r\}$. Thus, we have
$$d_\cH(\cB_{\mu_p},\cB_{\mu_q})\leq \sup_{y\in\sS^{d-1}}\left\|\frac{y}{\mu_p(y)}-\frac{y}{\mu_q(y)}\right\|_2\,.$$
Note that $y/\mu_p(y)$ (resp. $y/\mu_q(y)$) is in the unit sphere for the norm $\mu_p$ (resp. $\mu_q$).
Let us define $\mu_p^{min}=\inf_{x\in\sS^{d-1}}\mu_p(x)$. As the map $p\rightarrow \mu_p$ is uniformly continuous on the sphere $\sS^{d-1}$ (see Theorem 1.2 in \cite{GaretMarchandProcacciaTheret},) the map $p\rightarrow \mu_p^{min}$ is also continuous and $\mu^{min}=\inf_{p\in[p_0,1]} \mu_p^{min}>0$. Finally
\begin{align*}
 d_\cH(\cB_{\mu_p},\cB_{\mu_q})&\leq \sup_{y\in\sS^{d-1}}\left|\frac{1}{\mu_p(y)}-\frac{1}{\mu_q(y)}\right|\nonumber\\
 &\leq \sup_{y\in\sS^{d-1}}\frac{1}{\mu_q(y)\mu_p(y)}\left|\mu_p(y)-\mu_q(y)\right|\nonumber\\
 &\leq \sup_{y\in\sS^{d-1}}\frac{1}{(\mu^{min})^2}\left|\mu_p(y)-\mu_q(y)\right|\nonumber\\
 &\leq \frac{\kappa_d}{(\mu^{min})^2}(q-p)|\log(q-p)| \,.
 \end{align*}
This yields the result.
\end{proof}

\begin{rk} At this stage, we were not able to obtain Lipschitz continuity for $p\rightarrow \mu_p$. The difficulty comes from the fact that we do not know the correlation between $\gamma$ and the state of the boxes that $\gamma$ visits. At first sight, it may seem that the renormalization is responsible for the appearance of the log terms in Theorem \ref{heart}. However, when $p$ is very close to $1$, we can avoid renormalization and bypass $p$-closed edges at a microscopic scale as in \cite{CoxDurrett} but even in that case, we cannot obtain Lipschitz continuous regularity with the kind of combinatorial computations made in section \ref{ProbaEst}. A similar issue arises, it is hard to deal with the correlation between $p$-closed edges of $\gamma$ and the length of the microscopic bypasses. 
\end{rk}

\bibliographystyle{plain}
\def\cprime{$'$}

\end{document}